\documentclass{amsart}
\usepackage{amsmath}
\usepackage{amsfonts}
\usepackage{amssymb}
\usepackage{amsthm}
\usepackage{url}
\usepackage{dsfont}
\usepackage{enumitem}
\usepackage{graphicx}
\usepackage{caption}
\usepackage{subcaption}
\usepackage{comment} 
\usepackage[noadjust]{cite}
\usepackage{stmaryrd}
\usepackage{hyperref}

\numberwithin{equation}{section}

\newtheorem{prop}{Proposition}[section]
\newtheorem{lem}[prop]{Lemma}
\newtheorem{defn}[prop]{Definition}
\newtheorem{thm}[prop]{Theorem}

\newtheorem{rem}[prop]{Remark}
\newtheorem{cor}[prop]{Corollary}

\title{An investigation of stability on certain toric surfaces}
\author{Lars Martin Sektnan}
\address{D\'epartement de math\'ematiques\\
  Universit\'e du Qu\'ebec \`{a} Montr\'eal\\
  Case postale 8888, succursale centre-ville \\
  Montr\'eal (Qu\'ebec)\\
  H3C 3P8 \\
  Canada}
\email{lars.sektnan@cirget.ca}

\begin{document}
\maketitle
\begin{abstract}We investigate the relationship between stability and the existence of extremal K\"ahler metrics on certain toric surfaces. In particular, we consider how log stability depends on weights for toric surfaces whose moment polytope is a quadrilateral. We introduce a space of symplectic potentials for toric manifolds, which induces metrics with mixed Poincar\'e type and cone angle singularities. For quadrilaterals, we give a computable criterion for stability with $0$ weights along two of the edges of the quadrilateral. This in turn implies the existence of a definite log-stable region for generic quadrilaterals. This uses constructions due to Apostolov-Calderbank-Gauduchon and Legendre.

\end{abstract}
\section{Introduction}

The search for canonical metrics such as extremal K\"ahler metrics is a central topic in complex geometry. One of the key conjectures is the Yau-Tian-Donaldson conjecture relating the existence of extremal K\"ahler metrics in the first Chern class of a line bundle to algebro-geometric stability, the predominant stability notion being $K$-stability. There is also a version of this stability notion called relative log $K$-stability. Here one fixes a simple normal crossings divisor $D$ in a complex manifold and attaches non-negative weights to each irreducible component of $D$. For each choice of weights, one gets a different criterion for stability. 

In this article we study how relative log $K$-stability (with respect to toric degenerations) depends on weights, for certain toric surfaces. Toric varieties correspond to Delzant polytopes and the weights can be described by a measure on the boundary of the polytope. The stability condition we are considering therefore depends on the Delzant polytope and the boundary measure. However, this definition works equally well on any bounded convex polytope with such a boundary measure, regardless of whether it is Delzant or not, and we will work in this generality. Allowing any polytope, not just Delzant ones, features in Donaldson's continuity method for extremal metrics on toric varieties, see \cite{donaldson08}. There is also some geometric meaning for such polytopes, as they arise for toric Sasakian manifolds with irregular Reeb vector fields, see e.g. \cite{martellisparksyau06}, \cite{abreu10} and \cite{elegendre11b}.

Log stability is conjectured to be equivalent to the existence of an extremal metric with a mixture of singularities along the divisors corresponding to the facets of the polytope. For non-zero weights, the singularities are cone angle singularities with angle prescribed by the weight. The predominant behaviour along the edges with $0$ weight is expected to be Poincar\'e type singularities. However, we show that this is not the only behaviour one should expect for $0$ weights. 

A key to understanding the Yau-Tian-Donaldson conjecture is to understand what happens when an extremal metric does \textit{not} exist. For toric varieties, Donaldson conjectured in \cite[Conj. 7.2.3]{Don02} that there should be a splitting of the moment polytope into subpolytopes that each are semistable when attaching a $0$ measure to the sides that are not from the original moment polytope. In \cite{szekelyhidi08}, Sz\'ekelyhidi showed that such a splitting exists, under the assumption that the optimal destabilizer is a piecewise linear function. 

The subpolytopes in the splitting should come in two types. If the subpolytopes are in fact stable, they are conjectured to admit complete extremal K\"ahler metrics  on the complement of the divisors corresponding to the edges with vanishing boundary measure, whenever the subpolytopes are Delzant. If they are not stable, they are conjectured to be trapezia with no stable subpolytopes. Trapezia correspond to $\mathbb{CP}^1$-bundles over $\mathbb{CP}^1$, and this corresponds to the collapsing of an $S^1$ in the fibre over each point of this subpolytope, when trying to minimize the Calabi functional.

The work relates to several directions in K\"ahler geometry. Extremal K\"ahler metrics are solutions to a non-linear PDE, and explicit solutions are usually very difficult to find, even if one knows that such a metric exists. By using the constructions of Apostolov-Calderbank-Gauduchon and Legendre, we get explicit solutions to this PDE in \textit{ambitoric coordinates}. Also, the stability condition is often difficult to verify, and we find an easily computable criterion for the stability of a weighted quadrilateral with $2$ weights being $0$.

In general, the boundary measure attaches a non-negative weight to each facet of a polytope. For a quadrilateral $Q$ with weight vanishing along at least two edges, there is therefore a two parameter family of possible weights that we can attach to the remaining two edges, for each choice of edge pairs. Log stability on toric varieties is invariant under scaling of the weights, and so it therefore suffices to consider the weights $k_1, k_2$ such that $k_1 + k_2 = 1$. The main result of the article is the following theorem. See section \ref{genprops} for the conventions in the statement and the description of the numbers $r_0$ and $r_1$. 
\begin{thm}\label{mainquadthm} Let $E_i, E_j$ be two different edges of $Q$ that are not parallel. Then there exists \textnormal{explicit} numbers $0 \leq r_0 < r_1 \leq 1$ such that $(1-r) E_i + r E_j$ is
\begin{itemize} \item stable if $ r \in (r_0,r_1),$ 
\item not stable if $r \in [0,r_0)$ or $r \in (r_1,1]$.
\end{itemize}
Moreover, $(1-r) E_i + r E_j$ is
\begin{itemize} \item stable at $ r_0$ and $r_1$ if $E_i$ and $E_j$ are adjacent, unless $r_0=0$ or $r_1=1$, respectively,
\item not stable at $ r_0$ and $r_1$ if $E_i$ and $E_j$ are opposite.
\end{itemize}
If $E_i$ and $E_j$ are parallel, then $(1-r) E_i + r E_j$ is unstable for all $r \in [0,1]$.
\end{thm}
The condition defining $r_0$ and $r_1$ can be computed easily from the data of the weighted quadrilateral, see \cite[Sect. 4.5]{sektnanthesis} for explicit formulae. Note that while the proof of this theorem uses the ambitoric coordinates of \cite{ACG15}, the condition for stability can be expressed without mention of the ambitoric structure.

Our results also give some indications about the metrics one should expect to arise in Donaldson's conjecture on the splitting of a polytope into semistable subpolytopes. When allowing $0$ boundary measure, we show that the set of stable weights along the boundary is not always open. This is because the criterion defining the numbers $r_0,r_1$ in \ref{mainquadthm} is a closed condition in the case of boundary measures with $0$ weight along two adjacent edges.

This non-openness is unexpected, since stability is an open condition when all boundary measures are positive. We relate this phenomenon to the singular behaviour the metrics have along the divisors corresponding to the edges with $0$ boundary measure, see corollary \ref{metricsdescription}. This shows that there are several distinct asymptotics occuring for extremal potentials corresponding to weighted polytopes with $0$ boundary measure along some edges. This further indicates that one may expect several types of singular behaviour for the metrics in Donaldson's conjecture. In corollary \ref{ssquads}, we show that strictly semistable quadrilaterals admit a splitting into two stable subpolytopes.

The organisation of the paper is as follows.   We begin in section \ref{background} by recalling some background relating to toric varieties, stability and Poincar\'e type metrics. 

In section \ref{genprops}, we start considering the special case of quadrilerals and state some conventions and notation we will be using. In section \ref{formalsols}, we recall the ambitoric construction of Apostolov-Calderbank-Gauduchon in \cite{ACG14} and \cite{ACG15}. Their construction is phrased for rational data, but we note that it can be applied for quadrilaterals of non-Delzant type, with arbitrary non-negative boundary measure. This is no different, in \cite{ACG15} it has simply been stressed what one has to check in the ambitoric setting to ensure that the data corresponds to the moment polytope of a toric orbifold surface.

The main body of work is in section \ref{stabinv}, which is devoted to proving theorem \ref{mainquadthm} using the ambitoric framework. We find a definite stable region that generically splits the region that is unstable into $4$ connected components. Moreover, in contrast to when the boundary measure is positive on all edges of the quadrilateral, we show that the stable region is in general \textit{not open}.

In section \ref{relstometssection}, we relate our findings of the previous section to the question of existence of extremal metrics on the corresponding orbifold surface, whenever the quadrilateral is Delzant, shedding more light on \cite[Rem. 4]{ACG15}. In particular, we describe how the predominant behaviour of our solutions are of mixed cone singularity and Poincar\'e type singularities, in a weak sense. We show that the non-openness of the stable region when allowing the boundary measure to vanish on some edges is related to the existence of an extremal metric with singularities along a divisor, but that this singular behaviour is neither conical nor of Poincar\'e type.

\subsection*{Acknowledgements:} This work was done as a part of the author's PhD thesis at Imperial College London. I would like to thank my supervisor Simon Donaldson for his encouragement and insight. I gratefully acknowledge the support from the Simons Center for Geometry and Physics, Stony Brook University at which some of the research for this paper was performed. I would also like to thank Vestislav Apostolov for helpful comments.

\section{Background}\label{background}
We begin by recalling some of the background relevant to the article. The classification of toric varieties is discussed in subsection \ref{torvars}. In \ref{wtdstab}, we consider log $K$-stability for toric varieties and state it in the more general context of weighted convex polytopes. We also prove some basic properties that we will make use of in the particular case of quadrilaterals. The metrics we will mostly be concerned with later are Poincar\'e type metrics, whose definition we recall in \ref{genptmets}, before considering how they can be described in the toric setting in \ref{toricpt}. 

\subsection{Toric varieties}\label{torvars}

Toric varieties are compactifications of the complex $n$-torus $T^n_{\mathbb{C}} = (\mathbb{C}^*)^n$ admitting a holomorphic action of this torus extending the action on itself. From the symplectic point of view, one instead considers the action of the compact group $T^n = (S^1)^n$ and the space as a fixed symplectic manifold. Compact toric varieties are classified in terms of certain polytopes, called \textit{Delzant polytopes}.

\begin{defn} A toric symplectic manifold of dimension $2n$ is a symplectic $2n$-dimensional manifold $(M, \omega)$ with a Hamiltonian action of the $n$-torus $T^n$.
\end{defn}

There is then a moment map $\mu : M \rightarrow (\mathfrak{t}^n)^*$ for the torus action. The image of the moment map $\mu$ is the convex hull of the fixed points of the action, provided $M$ is compact. Only a certain type of images appear. The following definitions will capture precisely the type of image occuring in the compact case. Recall that a half-space $H$ in a vector space $V$ is a set of the form $\{ x \in V : l(x) \geq 0 \}$ for some affine function $l : V \rightarrow \mathbb{R}$. Its boundary $\partial H$ is the set $\{ x \in V : l(x) = 0 \}$.

\begin{defn}\label{delzpoldefn} A convex polytope $\Delta$ in a finite dimensional vector space $V$ is a non-empty intersection $\cap_{i = 1}^k H_i$ of finitely many half-spaces $H_i$.  A \textnormal{face} of $\Delta$ is a non-trivial intersection
\begin{align*} F = \Delta \cap \partial H
\end{align*}
for some half-space $H$ such that $\Delta \subseteq H$. If $H$ is unique, then $F$ is called a \textnormal{facet}.

Let $\langle \cdot , \cdot \rangle$ denotes the contraction $V \times V^* \rightarrow \mathbb{R}$. Given a lattice $\Lambda$ in $V^*$, a polytope $\Delta$ is called \textnormal{Delzant}, with respect to this lattice, if it is bounded and can be represented as
\begin{align*} \Delta = \bigcap_{i=1}^k \{ x \in V : \langle x , u_i \rangle  \geq c_i \}
\end{align*}
where each $u_i \in \Lambda$ and each $c_j \in \mathbb{R}$, and moreover that each vertex is an intersection of exactly $n$ facets $F_i$ such that the $u_i$ form a basis of the lattice over $\mathbb{Z}$, where $n$ is the dimension of $V$.
\end{defn}

We make some remarks and mention some language we will use. We will call the $u_i$ appearing in the definition of a facet $F_i$ the \textit{conormal} to $F_i$. This is not unique, but we can fix it as follows. An element $u$ of the lattice $\Lambda$ is called \textit{primitive} if $\lambda u \in \Lambda$ for some $|\lambda | \leq 1$ implies that $\lambda = \pm 1$. So up to sign, there is a unique multiple $u$ of $u_i$ which is primitive. We can fix the sign of $u$ by requiring that $\Delta \subseteq \{ x : \langle x , u \rangle \geq c_i \}$. We then say that $u$ is \textit{inward-pointing}. 

The classification theorem for symplectic toric manifolds says that they are classified by Delzant polytopes.

\begin{thm}[\cite{delzant88}]\label{toricclassification} 

Let $(M, \omega)$ be a symplectic toric manifold and let $\mu$ be a moment map for the torus action. Then the image $\mu (M)$ of $M$ is a Delzant polytope in $(\mathfrak{t}^n)^*$ with respect to the integer lattice in $\mathfrak{t}^n = \mathbb{R}^n$.  Isomorphic symplectic toric manifolds give isomorphic Delzant polytopes and moreover, for each Delzant polytope $P$, there exists a toric symplectic manifold $(M_P, \omega_P)$ with a moment map whose image is $P$.
\end{thm}

Note here that $\mu$ maps to $(\mathfrak{t}^n)^*$, so in terms of the definition \ref{delzpoldefn}, we have $V = (\mathfrak{t}^n)^*$, $V^* = \mathfrak{t}^n$ and $\Lambda = \mathbb{Z}^n = \text{ker} ( \text{exp} : \mathfrak{t}^n \rightarrow T^n )$. The theorem was extended to the orbifold case by Lerman-Tolman in \cite{lermantolman97}.

All compact toric symplectic manifolds are obtained as the symplectic reduction of a torus $T^d$ acting on $\mathbb{C}^d$, for some $d$. One then takes the quotient by a subtorus $N = T^{d-n}$ and is left with a quotient space $M_P$ on which an n-torus $T^n = T^d/N$ acts in a Hamiltonian fashion. This gives a construction of $(M_P, \omega_P)$, the toric manifold associated to a polytope $P$.

To construct the manifold above (ignoring the symplectic form), we could instead have started with a complex point of view, where we would have everything complexified. That is, we would work with the complexified groups $N_{\mathbb{C}} \cong T_{\mathbb{C}}^{d-n}, T^d_{\mathbb{C}}$ and $T^n_{\mathbb{C}}$ and taken a quotient $\mathbb{C}^d \sslash N_{\mathbb{C}}$, the GIT quotient. As a smooth manifold, these are diffeomorphic, but the symplectic quotient comes with a symplectic structure and the GIT quotient comes with a complex structure.

\begin{rem} The complex quotient $\mathbb{C}^d \sslash N_{\mathbb{C}}$ does not depend on which moment map we chose for the action on the resulting smooth manifold. That is, it does not depend on translations of $P$. In fact, more is true. Different polytopes can give rise to the same manifold (the complex quotient only depends on the \lq \lq fan" of $P$, which in the compact case is the arrangement of the conormals of $P$ in the lattice). The significance is that the polytope contains more information than the complex picture, we have also specified a cohomology class $\Omega = [\omega] \in H^2(M, \mathbb{R})$. This cohomology class turns out to be integral if and only if, after a translation, the vertices of $P$ lie on the lattice.
\end{rem}

\subsection{Weighted stability}\label{wtdstab}

Let $d\lambda$ be the Lebesgue measure. We say a measure $d\sigma$ on the boundary of a bounded convex polytope $P$ is a \textit{positive boundary measure for $P$} if on the $i^{\textnormal{th}}$ facet $F_i$ of $P$, $d\sigma$ satisfies
\begin{align}\label{coneanglemeasurepos} l_i \wedge d\sigma = \pm r_i d \lambda
\end{align}
where $l_i$ is an affine function defining $F_i$ and the $r_i > 0$ are constants. We say $d \sigma$ is \textit{non-negative} if we relax the condition to $r_i \geq 0$, only. If $d\sigma$ is a non-negative boundary measure on $P$, we call the pair $(P,d\sigma)$ a \textit{weighted polytope}. 

Note that if $P$ is Delzant, then there is a canonical associated boundary measure. This is given by satisfying \ref{coneanglemeasurepos} with $r_i =1$ and the $l_i$ being the defining functions
\begin{align*} l_i (x) = \langle u_i, x \rangle + c_i,
\end{align*}
where $c_i \in \mathbb{R}$ and $u_i$ is the primitive inward-pointing normal to the facet $l_i^{-1} (0) \cap P$.

Let $A$ be a bounded function on a bounded convex polytope $P$. One can then define a functional $\mathcal{L}_A$ on the space of continuous convex functions on $P$ by 
\begin{align} \mathcal{L}_A (f) = \int_{\partial P} f d \sigma -  \int_{P} A f d \lambda .
\end{align}
Note that there is a unique affine linear $A$ such that $\mathcal{L}_{A} (f) = 0$ for all affine linear $f$. 
\begin{defn}\label{assafffn} Given a weighted polytope $(P,d\sigma)$, we call the affine linear function $A$ such that $\mathcal{L}_{A}$ vanishes on all affine linear functions the \textnormal{affine linear function associated to the weighted polytope} $(P,d\sigma)$. Also, we write $\mathcal{L} = \mathcal{L}_{A}$.
\end{defn}

We say a function $f$ on $P$ is \textit{piecewise linear} if it is the maximum of a finite number of affine linear functions. We say it is \textit{rational} if the coefficients of the affine linear functions are all rational, up to multiplication by a common constant. 

\begin{defn}\label{weightedpolytopestability} Let $(P, d\sigma)$ be a weighted polytope. We say $P$ is \textnormal{weighted polytope stable}, or more briefly \textnormal{stable}, if 
\begin{align}\label{stabdefnineq} \mathcal{L} (f) \geq 0 
\end{align}
for all piecewise linear functions $f$, with equality if and only if $f$ is affine linear. If $(P, d\sigma)$ is not stable, we say it is \textnormal{unstable}. If \ref{stabdefnineq} holds for all piecewise linear $f$, but there is a non-affine function $f$ with $\mathcal{L} (f) = 0$, we say $(P, d\sigma)$ is \textnormal{strictly semistable}. We say $(P,d\sigma)$ is \textnormal{semistable} if it is either stable or strictly semistable.
\end{defn}

\begin{rem} If $P$ is Delzant and $d\sigma$ is the canonical boundary measure associated to $P$, then this is the definition of $K$-stability with respect toric degenerations, see \cite{Don02}.
\end{rem}

A natural question one could ask is given a bounded convex polytope $P$, how does stability depend on the weight $d\sigma$? By specifying a positive background measure $d\sigma^0$, we identify the set of weights with $\mathbb{R}^d_{\geq 0} \setminus \{ 0 \}$. We now give two elementary lemmas about the set of stable weights.

\begin{lem}\label{setofstblwtslem} Let $(P, d\sigma^0 )$ be a polytope with $d$ facets $F_1, \cdots, F_d$, and with $d\sigma^0$ an everywhere positive  measure on the boundary $\partial P$ of $P$ as above. Then the set of weights $\underline{r} = (r_1, \cdots, r_d) \in \mathbb{R}^{d}_{\geq 0}$ such that $(P, d\sigma_{\underline{r}})$ is stable is a convex subset of $\mathbb{R}^d_{\geq 0}$.
\end{lem}

\begin{proof} Let $\underline{r}_0, \underline{r}_1$ be stable weights, and set $\underline{r}_t = (1-t) \underline{r}_0 + t \underline{r}_1$. Let $A_t$ be the affine function associated to the weighted polytope $(P, d\sigma_{\underline{r}_t})$. Then $A_t = (1-t) A_0 + t A_1$, and so, for all convex functions $f$ on $P$, we have
\begin{align*} \mathcal{L}_{\underline{r}_t} (f) &= (1-t) \mathcal{L}_{\underline{r}_0} (f) + t \mathcal{L}_{\underline{r}_1} (f) \\
& \geq 0.
\end{align*}
Moreover, since $\mathcal{L}_{\underline{r}_t} (f) \geq 0$ with equality if and only if $f$ is affine for $t = 0,1$, it follows that this holds for all $t \in [0,1]$ too. Hence $\underline{r}_t$ is a stable weight for all $t \in [0,1]$, as required.
\end{proof}

\begin{lem} Let $\underline{r}$ be a stable weight for $(P, d\sigma^0)$. Then $c \cdot \underline{r}$ is a stable weight for all $c > 0$. 
\end{lem}

\begin{proof} $A_{c \cdot \underline{r}} = c A_{\underline{r}}$, and so $\mathcal{L}_{c \underline{r}} = c \cdot \mathcal{L}_{\underline{r}}$. The lemma follows immediately from this.
\end{proof}

The stable set in $\mathbb{R}^d$ is thus a convex cone on the stable weights with $\sum_i r_i = 1$, and so to fully describe the stable set one can without loss of generality consider weights such that $\sum_i r_i = 1$. Later we investigate the dependence of stability on the weights in the particular case of quadrilaterals.

We end the section with two important lemmas. For the first, let $\text{SPL} (P)$ denote the space of \text{simple} piecewise linear functions on $P \subseteq (\mathbb{R}^2)^*$, that is functions $f$ of the form $x \mapsto \text{max} \{ 0, h(x) \}$ for an affine linear function $h : (\mathbb{R}^2)^* \rightarrow \mathbb{R}$. Note that we have a map 
\begin{align*}\textnormal{SPL} : \textnormal{Aff}(\mathbb{R}^2) \rightarrow \textnormal{SPL} (P)
\end{align*}
given by
\begin{align*} h \mapsto \text{max} \{ 0 , h(x) \}.
\end{align*}

Let $P$ be a $2$-dimensional polytope and fix two edges $E_1$ and $E_2$ of $P$ with vertices $v_1, w_1$ and $v_2, w_2$, respectively. Any point $p$ on $E_1$, respectively $q$ on $E_2$, can then be written as
\begin{align*}  p &= (1-s) v_1 + s w_1 ,\\
q&= (1-t) v_2 + t w_2
\end{align*}
for some $s,t \in [0,1]$. Let $p_i$, respectively $q_i$, be the $i^{\textnormal{th}}$ component of $p$, respectively $q$. This determines an affine linear function $l_{s,t}$ which vanishes on $p,q$ and for which the coefficients for the non-constant terms are linear in $s$ and $t$. Specifically, writing $l_{s,t} = a x + by + c$, let 
\begin{align*} a &= q_2 - p_2, \\
b&= p_1 - q_1, \\
c &= - a p_1 - b p_2.
\end{align*}
We then have
\begin{lem}\label{polylem} Let $\phi : [0,1] \times [0,1] \rightarrow \mathbb{R}$ be given by 
\begin{align*} (s,t) \mapsto \mathcal{L} (\textnormal{SPL} (l_{s,t}) ).
\end{align*}
Then $\phi$ is a polynomial in $(s,t)$ of bidegree $(3,3)$ and total degree $5$.
\end{lem}
For the proof the integrals one has to perform, say the ones over the polytope, can be decomposed as the integral of $A l_{s,t}$ over some fixed region $R$, where this statement holds, and a quadilateral region $Q_{s,t}$ bounded by $E_1,E_2, l_{0,0}^{-1} (0)$ and $l_{s,t}^{-1} (0)$. Direct computation, which we omit, then shows that this holds. Similarly for the boundary region.

The second lemma we will need concerns edges with $0$ weight in weighted two dimensional polytopes.

\begin{lem}\label{critpts0bdry} Let $P$ be a $2$-dimensional polytope with non-negative boundary measure $d\sigma$. Let $\nu(s,t)$ be the polynomial in \ref{polylem} for two edges $F_1$ and $F_2$ adjacent to an edge $E$ along which $d\sigma$ vanishes. Then the point in $[0,1] \times [0,1]$ corresponding to a simple piecewise linear function with crease $E$ is a critical point of $\nu$.
\end{lem}
\begin{proof}We may assume that $E= E_1$, $F_1 = E_4$ and $F_2 = E_2$, where $E_1, E_2, E_4$ are the edges specified at the beginning of this section. Note that $P$ need not have $4$ edges, but we can take the three edges we are considering to be of this form. The boundary measure vanishes along $E$, is $r_1 dy$ along $F_1$ and $r_2 dx$ along $F_2$ for some non-negative constants $r_1, r_2$. The affine function $l_{s,t}$ that we integrate in \ref{polylem} is 
\begin{align*} l_{s,t} (x,y) &= (sq - tk) x - (1+ sp)y + tk (1+ sp)
\end{align*}
and the point corresponding to the crease being the edge $E$ is $(0,0)$.

By linearity it suffices to show that the directional derivative in two independent directions vanish. We first consider the partial derivative $\frac{\partial \phi}{\partial t} (0,0)$. So we are letting $s=0$ and we would like to compute the derivative of 
\begin{align*} t \mapsto \int_{\partial P \cap \{ l_t \geq 0\}} l_t d\sigma - \int_{P \cap\{ l_t \geq 0\}} l_t  A d\lambda
\end{align*}
at $0$, where $l_t=tk - y - t k x $. Here $A$ is the affine linear function associated to the weighted polytope $(P,d\sigma)$.

In taking the integral over the polytope, the integrals of all the terms in $l A$ is always divisible by $t^2$, since the constant and $x$-term in $l$ has a factor of $t$, and all integrals involving $y$ will introduce an extra factor of $t$. Thus the derivative of $\int_{P \cap l>0} l  A d\lambda$ is $0$ and we only need to consider the terms coming from the integral over the boundary. 

For this part, we are then considering the derivative of 
\begin{align*} t \mapsto \int_{F_1 \cap \{ l_t \geq 0\}} l_t d\sigma 
\end{align*}
since $l_t$ is only positive on $E$ and $F_1$, and the boundary measure vanishes on $E$. But this equals
\begin{align*} r_1 \int_{0}^{tk} (tk - y) dy
\end{align*}
since $F_1 \subseteq \{ x = 0\}$ and so $l_t = tk -y$ on $F_1$. It follows that the derivative of this function vanishes at $t=0$.

To complete the proof we need to check that the directional derivative in a linearly independent direction vanishes. One can consider $\frac{\partial \phi}{\partial s} (0,0)$. This case is similar. It then follows that $(0,0)$ is a critical point. 
\end{proof}

\subsection{Poincar\'e type metrics}\label{genptmets}

Consider the punctured unit (open) disk $B_1^* \subseteq \mathbb{C}$ with the metric
\begin{align}\label{localcusp} \frac{|dz|^2}{(|z|  \log |z|)^2}.
\end{align}
Here we use the notation $|dz|^2 = dx^2 + dy^2$, where $z = x + iy$. This is the standard cusp or Poincar\'e type metric on $B_1^*$. The associated symplectic form is
\begin{align}\label{stdcuspkahlerpotential} \frac{i dz \wedge d\overline{z}}{ |z|^2 \log^2 (|z|)} = 4  i \partial \overline{\partial} ( \log (- \log ( |z|^2) ) ).
\end{align}
Poincar\'e type metrics are K\"ahler metrics on $X \setminus D$ which near $D$ look like the product of the Poincar\'e type metric on $B_1^*$ with a metric on $D$. These metrics have a rich history of study. A central result is the existence of K\"ahler-Einstein metrics with such asymptotics, analogous to Yau's theorem in the compact case, by Cheng-Yau, Kobayashi and Tian-Yau in \cite{chengyau80}, \cite{rkobayashi84} and \cite{tianyau87}, respectively.

Auvray made a general definition of metrics with such singularities along a simple normal crossings divisor $D$ in a compact complex manifold $X$. That $D$ is simple normal crossings means that we can write $D = \sum_k D_k$, where each $D_k$ is smooth and irreducible, and the $D_{k}$ intersect transversely in the sense that for each choice $k_1, \cdots , k_l$ if distinct indices, we can around each point in $D_{k_1} \cap \cdots \cap D_{k_l}$ find a holomorphic chart $(U, z_1, \cdots, z_n)$ such that $D_{k_j} \cap U = \{ z_{j} = 0 \} \cap U$. Note that in particular $l$ is at most the dimension of $X$. Note that on each such chart $U$, we have a standard locally defined cusp metric whose associated $2$-form is given by 
\begin{align*} \omega_{cusp} = \sum_{j=1}^l \frac{i dz_j \wedge d \overline{z}_j }{(|z_j|  \log |z_j|)^2} + \sum_{j > l} i dz_j \wedge d \overline{z}_j  .
\end{align*}

Given such a divisor, one can for each $k$ define a model function $f_k$, which when patched together gives the model K\"ahler potential for a Poincar\'e type metric. More precisely, fix a holomorphic section $\sigma_k$ of $\mathcal{O} ( D_k)$ such that $D_k$ is the zero set of $\sigma_k$. Also fix a Hermitian metric $|\cdot|_k$ on $\mathcal{O} (D_k)$, which we assume satisfies $ | \sigma_k |_k \leq e^{-1}$. Thus, for each $\lambda$ sufficiently large, the function $f_k = \log ( \lambda - \log ( |\sigma_k|_k^2 ) )$ is defined on $X \setminus D_k$. 

Let $\omega_0$ be a K\"ahler metric on the whole of the compact manifold $X$. By the above we can, for sufficiently large $\lambda$, pick $A_k > 0$ such that if $f = \sum_k A_k f_k $, then $\omega_{f} = \omega_0 - i \partial \overline{\partial} f$ is a positive $(1,1)$-form on $X \setminus D$. Poincar\'e type metrics are then metrics on $X \setminus D$ defined by a potential with similar asymptotics to $f$ near $D$.
\begin{defn}[{\cite[Def. 0.1]{auvray14c},\cite[Def. 1.1]{auvray13}}]\label{poincaretypemetdefn} Let $X$ be a compact complex manifold and let $D$ be a simple normal crossings divisor in $X$. Let $\omega_0$ be a K\"ahler metric on $X$ in a class $\Omega \in H^2(X, \mathbb{R})$. A smooth, closed, real $(1,1)$ form on $X \setminus D$ is a Poincar\'e type K\"ahler metric if 

{
\itemize

\item $\omega$ is quasi-isometric to $\omega_{cusp}$. That is, for every chart $U$ as above, and every compact subset $K$ of $B_{\frac{1}{2}} \cap U$, there exists a $C$ such that throughout $K$, we have
\begin{align*} C \omega_{cusp} \leq \omega \leq C^{-1} \omega_{cusp}.
\end{align*} 
}

Moreover, the class of $\omega$ is $\Omega$ if
\itemize
\item $\omega = \omega_0 + i \partial \overline{\partial} \varphi$ for a smooth function $\varphi$ on $X \setminus D$ with $| \nabla_{\omega_f}^j \varphi |$ bounded for all $j \geq 1$ and $\varphi = O(f)$.
\end{defn}

\subsection{Poincar\'e type metrics on toric varieties}\label{toricpt}

From the works of Guillemin and Abreu in \cite{guillemin94jdg} and \cite{abreu98}, respectively, one can describe all $T^n$-invariant K\"ahler metrics in a given K\"ahler class through a space of strictly convex functions on the associated moment polytope. We will now describe how one can extend this to the case of metrics with mixed Poincar\'e and cone angle singularities along the torus-invariant divisors of a toric manifold. 

One way to view the correspondence between $T^n$-invariant K\"ahler metrics on compact toric manifolds and certain strictly convex functions on $P$ is the following. For each strictly convex function $u$ on $P$ which is smooth on $P^{\circ}$, the Legendre transform induces a map $\psi_u : P^{\circ} \times T^n \rightarrow (\mathbb{C}^*)^n$, which we can then think of as a map between the free orbits in the symplectic quotient $M_P$ and the complex quotient $N_P$ associated to $P$, respectively. The Guillemin boundary conditions for the function $u$ are the precise boundary conditions such that $\psi_u$ extends as a diffeomorphism $M_P \rightarrow N_P$ taking $[\omega_P] \in H^2(M_P, \mathbb{R})$ to $\Omega_P \in H^2(N_P, \mathbb{R})$. 

It will be convenient to encode the data of the singularities in a boundary measure again. Given a positive boundary measure there is a unique $l_i$ such that $P \subseteq l_i^{-1} ([0, + \infty) )$ for all $i$, and that \ref{coneanglemeasurepos} is satisfied with $r_i = 1$. We call the collection $l_1, \cdots, l_d$ the \textit{canonical defining functions} of $(P,d\sigma)$. For any bounded convex polytope $P$, we then define a space of symplectic potentials. 
\begin{defn}\label{coneanglepotentials} Let $P$ be a bounded convex polytope and let $d\sigma$ be a positive boundary measure for $P$. Let $l_i$ be the canonical defining functions for $(P,d\sigma)$. We define the space of symplectic potentials $\mathcal{S}_{P, d \sigma}$ to be the space of strictly convex functions $u \in C^{\infty} (P^{\circ}) \cap C^0 (P)$ satisfying
\begin{align*} u = \frac{1}{2} \sum_i l_i \log l_i + h,
\end{align*}for some $h \in C^{\infty} (P)$ and which further satisfies that the restriction of $u$ to the interior of any face of $P$ is strictly convex.
\end{defn}

In the case when $P$ is Delzant, $\mathcal{S}_{P, d \sigma}$ then precisely describes metrics with cone angle singularities along the torus-invariant divisors, the cone angle being prescribed by $d\sigma$. 
\begin{prop}[{\cite[Prop. 2.1]{datarguosongwang13}}] Let $P$ be a Delzant polytope with canonical measure $d \sigma^{0}$. Let $d\sigma$ be a positive boundary measure for $P$, so on each facet $F_i$ of $P$, $d\sigma$ satisfies 
\begin{align*} d\sigma_{| F_i} = r_i d\sigma^0_{| F_i}
\end{align*}
for some $r_i > 0$. Then through the Legendre transform, symplectic potentials $u \in \mathcal{S}_{P, d \sigma}$ induce metrics with cone singularities along the torus invariant divisors $D_i$ corresponding to the facets $F_i$. The cone angle singularity along $D_i$ is $2 \pi r_i$.
\end{prop}

Let $D$ be a not neccessarily irreducible torus-invariant divisor in $N_P$, so $D$ is a union of some of the $D_i$ as above. The goal of this section is to instead describe the precise conditions on the function $u$ such that $\psi_u$ induces a diffeomorphism such that $ (\psi_u^{-1})^* (\omega_P)$ is a metric on $N_P \setminus D$ with Poincar\'e type singularities along $D$ and cone angle single singularities along the remaining torus-invariant divisors. 

The model cusp metric on the unit punctured disk in $\mathbb{C}$ has associated K\"ahler form given by
\begin{align*} \frac{i dz \wedge d\overline{z}}{|z|^2 \log^2 ( |z|^2) } . 
\end{align*}
It is induced by the Legendre transform of the function
\begin{align*} - \log (x).
\end{align*}
This motivates the definition below of the space of Poincar\'e type metrics.

Let $P$ be a Delzant polytope with facets $F_1, \cdots, F_d$. We let $(N_P, \Omega_P)$ be the corresponding complex manifold and K\"ahler class associated to $P$, and let $D_i$ be the divisor in $N_P$ corresponding to the facet $F_i$. Suppose $d\sigma$ is a non-negative boundary measure for $P$. Let $\{i_1, \cdots , i_k\}$ be the subset of $\{ 1, \cdots , d \}$ on which $r_i$ vanishes, which, after relabelling of the $F_i$, we will assume is $1, \cdots, k$. Then we let $D$ denote the divisor $D_{1} + \cdots + D_{k}$ corresponding to the facets on which $d\sigma$ vanishes. 

Given a non-negative boundary measure $d\sigma$ for $P$, let $d\tilde{\sigma}$ be a positive boundary measure which agrees with $d\sigma$ on the facets where $d\sigma$ does not vanish. For a symplectic potential $v \in \mathcal{S}_{P, d \tilde{\sigma}}$ and positive real numbers $a_{1}, \cdots, a_{k} > 0$, define $ u_{\underline{a},v} : P^{\circ} \rightarrow \mathbb{R}$ by
\begin{align}\label{modelptpotentialeqn} u_{\underline{a}, v} = v + \sum_{i= 1}^k( -a_i \log l_i ).
\end{align}

In the author's thesis \cite{sektnanthesis}, it was shown that potentials of this form induce metrics with Poincar\'e type singularities along $D$. More precisely,
\begin{prop}\label{modelptpotential} Let $(P,d\sigma)$ be a weighted Delzant polytope, where $d\sigma$ is a non-negative boundary measure. Then through the Legendre transform, $u_{\underline{a},v} $ defines a K\"ahler metric on $N_P \setminus D$ with mixed Poincar\'e and cone angle singularities in the class $\Omega_P$. The Poincar\'e type singularity is along $D$, and the cone angle singularities are along the divisors $D_i$ with $i > k$, the cone angle singularity along $D_i$ being equal to that of the metric induced by $v$.
\end{prop}

This serves as model Poincar\'e type potentials. More generally, the space of $T^n$-invariant Poincar\'e type metrics in a given class can be described by functions satisfying the following definition.

Also, recall that associated to $d\tilde{\sigma}$ there is a canonical choice of defining functions $l_i$ for $P$, whose zero sets intersect $P$ in facets $F_i$. For a non-negative boundary measure $d\sigma$ we can get canonical defining functions for the $i$ such that $d\sigma_{|F_i} \neq 0$ by the same requirement on these facets. 

For the functions $u$ and $u_{\underline{a},v}$ below we will let $U$ and $U_{\underline{a},v}$ denote their respective Hessians. Given a non-negative boundary measure $d\sigma$, we let $d_{PT} : P \rightarrow \mathbb{R}$ be a positive function on $P$ which is smaller than $1$ everywhere, and which agrees with the distance function to the Poincar\'e type facets near these facets. The Poincar\'e type facets are the facets on which $d\sigma$ vanishes.

\begin{defn}\label{toricptpotentialsdefn} Let $P$ be a polytope with facets $F_1, \cdots, F_d$ and let $d\sigma$ be a non-negative boundary measure. Let $l_i$ be the canonical defining functions for the $i$ such that $d\sigma$ does not vanish along $F_i$. Define $\mathcal{S}_{P,d \sigma}$ to be the space of smooth strictly convex functions $u : P^{\circ} \rightarrow \mathbb{R}$ that can be written as
\begin{align}\label{uismodelplush} u = \frac{1}{2} \sum_{i : d\sigma_{|F_i} \neq 0}l_i \log l_i + h,
\end{align}
for some $h \in C^{\infty} (  P \setminus \cup_{i : d\sigma_{|F_i} = 0} F_i )$, and which moreover satisfy that there is a model potential $u_{\underline{a},v}$ for $(P, d\sigma)$ such that
\begin{itemize} \item $u$ restricted to each facet where the boundary measure does not vanish is strictly convex,
\item $| u | \leq C ( - \log (d_{PT}))$ for some $C>0$,
\item there is a $c >0$ such that 
\begin{align}\label{potentialinequality} c^{-1} U_{\underline{a}, v}\leq  U  \leq c  U_{\underline{a}, v},
\end{align}
\item for all $i \geq 1$, we have that $| \nabla^i u |_{u_{\underline{a},v}}$ and $| \nabla^i u_{\underline{a},v} |_{u_{\underline{a},v}}$ are mutually bounded.
\end{itemize}
Here $\nabla u = \nabla^1 u$ is the gradient of $u$ with respect to $u_{\underline{a},v}$, $\nabla^i$ denotes the higher derivatives with respect to the Levi-Civita connection of $u_{\underline{a},v}$ and $|\cdot|_{u_{\underline{a},v}}$ denotes the norm on the higher tensor bundles of $TP^{\circ}$ with respect to $u_{\underline{a},v}$.
\end{defn}

For a Delzant polytope, elements of $\mathcal{S}_{P, d\sigma}$ also give K\"ahler metrics with mixed Poincar\'e type and cone singularities. For the proof, see \cite[Prop. 3.10]{sektnanthesis}.

\begin{prop}\label{allptpotsprop} Suppose $P$ is a Delzant polytope and let $d\sigma$ be a non-negative boundary measure. Then for all $u \in \mathcal{S}_{P, d\sigma}$, $u$ defines through the Legendre transform a K\"ahler metric on $N_P$ in the class $\Omega_P$ with mixed Poincar\'e type and cone angle singularities, the singularity being prescribed by $d\sigma$. 

Conversely, if $\omega \in \Omega_P$ is the K\"ahler form of a $T^n$-invariant metric on $N_P$ of Poincar\'e type along a torus-invariant divisor $D$, then it is induced by a function $u$ on $P^{\circ}$ satisfying definition \ref{toricptpotentialsdefn}.
\end{prop}

\section{Statement of result}\label{genprops}
We now come to the main part of the article, where we investigate weighted polytope stability for quadrilaterals. We begin by stating the conventions we will use.

Let $Q$ be the quadrilateral with vertices  $v_1 = (0,0), v_2 = (1,0), v_3 = (1+p,q)$ and $v_4=(0,k)$, for some $q,k>0$ and $p> \textnormal{max } \{ - \frac{q}{k}, -1 \}$. Then $Q$ is a convex quadrilateral, and all quadrilaterals can be mapped to such a quadrilateral via a translation and a linear transformation. When the parameters are rational, this is a rational Delzant polytope, and so corresponds to a toric orbifold surface $X_Q$. Since it is a quadrilateral, $b_2 (X_Q) = 2$.\footnote{In general, a two dimensional rational Delzant polytope with $d$ edges is the moment polytope of a toric orbifold surface with $b_2 = d-2$.} The edges $E_1, \cdots, E_4$ of $Q$ are given as $l_i^{-1} (0)$, where 
\begin{align*} l_1(x,y) &= y,\\
 l_2(x,y) &= -q x + p y + q,\\
 l_3(x,y) &= (q-k)x- (1+p)y + k(1+p),\\
 l_4(x,y) &= x,
\end{align*}
and $Q = \bigcap_i l_i^{-1}([0, \infty) )$. The canonical measure $d\sigma$ on $\partial Q$ associated to these defining equations is thus given by 
\begin{align*} d \sigma_{| E_1} &= dx,\\
d \sigma_{| E_2} &= \frac{1}{q} dy,\\
d \sigma_{| E_3} &= (1+p) dx,\\
d \sigma_{| E_4} &= dy.
\end{align*}
We will identify the weight $\underline{r} = (r_1, \cdots, r_4)$, and so also the corresponding measure, with a formal sum $\sum_i r_i E_i$. Thus, for example, $(\frac{1}{2},\frac{1}{2},0,0)$ is identified with $\frac{1}{2} E_1 + \frac{1}{2} E_2$.

The property describing the explicit numbers $r_0$ and $r_1$ of theorem \ref{mainquadthm} is the following. $Q$ has two pairs of opposite sides. Let $\phi (s,t)$ and $\psi(s,t)$ denote the polynomials of lemma \ref{polylem} for these two pairs of edges. The domains of these functions each have two points which correspond to affine functions, i.e. the crease is exactly an edge of $Q$. These points are opposite vertices of $[0,1] \times [0,1]$, and, after possibly replacing e.g. $\phi(s,t)$ with $\phi(s,1-t)$, we can take these to be $(0,0)$ and $(1,1)$. 

Similarly, we can also assume that in the case when $d\sigma$ vanishes on two adjacent sides, $(0,0)$ in each domain is the point corresponding to the crease being on an edge with vanishing boundary measure. In the case of $d\sigma$ vanishing on two opposite edges, we can assume $\phi$ is the function parameterising the Donaldson-Futaki invariant of simple piecewise linear functions with crease along the \textit{other} pair of opposite edges. In particular, the points $(0,0)$ and $(1,1)$ in the domain of $\phi$ correspond to simple piecewise linear functions with crease on an edge where $d\sigma$ vanishes.

In fact, lemma \ref{critpts0bdry} implies that the points where the corresponding crease is an edge with vanishing boundary measure are critical points of $\phi$ or $\psi$. In particular, we get that the vanishing of the determinant at such a point is an invariant notion, i.e. it does not depend on the scale we used in defining $\phi$ and $\psi$. 

\begin{lem} Let $d\sigma_r$ be the boundary measure for $Q$ corresponding to $rE_i + (1-r) E_j$ for edges $E_i,E_j$ of $Q$, and let the polynomials of \ref{polylem} for this boundary measure be $\phi_r$ and $\psi_r$. Then the determinant of the Hessian of $\phi_r$ or $\psi_r$ at a point in $[0,1] \times [0,1]$ is quadratic in $r$.
\end{lem}
\begin{proof} From their definition and lemma \ref{setofstblwtslem}, $\phi_r$ and $\psi_r$ are linear in $r$, and hence so are all their second derivatives with respect to $s$ and $t$. Hence the determinant is of degree $2$ in $r$.
\end{proof}

We can then finally characterise what the $r_0$ and $r_1$ in theorem \ref{mainquadthm} are. They are given as the end-points of the intersection of the two regions where $\phi_r$ and $\psi_r$ have \textit{non-negative} determinant at the points corresponding to simple piecewise linear functions with crease an edge with $0$ boundary measure. As remarked above, this does not depend on our choice of scale for $\phi$ and $\psi$. 

In \cite[App. B]{ACG15}, Apostolov-Calderbank-Gauduchon showed that unless $Q$ is a parallelogram, it has both unstable and stable weights, when all weights are positive. From theorem \ref{mainquadthm}, we also get a result about the set of unstable weights for quadrilaterals, now allowing weights to be $0$. The vertices of $\sum r_i = 1$, corresponding to measures supported on one edge only, are always unstable. Thus the set of unstable weights can have at most four connected components. This is generically the case, but in the case of parallel sides there is different behaviour. Specifically, we have the following.

\begin{cor} Let $Q$ be a quadrilateral. Then the number of connected components of the unstable set is
\begin{itemize} \item $4$ if $Q$ has no parallel sides,
\item $3$ if $Q$ is a trapezium which is not a parallelogram,
\item $2$ if $Q$ is a parallelogram.
\end{itemize}
\end{cor}
\begin{proof} If $Q$ has no parallel sides, then theorem \ref{mainquadthm} implies that there is a stable weight on each edge of the $3$-simplex $\sum_i r_i = 1$. Since the stable set is a convex set, it follows that the stable set contains a sub-simplex whose complement has $4$ connected components. Thus the unstable set does too. 

If $Q$ is a trapezium, but not a parallelogram, then the weights along the edge corresponding to weights which are non-zero only on the two parallel sides are all unstable. This reduces the number of connected components by one. 

Finally, if $Q$ is a parallelogram, the unstable set is precisely the two edges of the simplex $\sum r_i$ corresponding to having zero weights on two opposite edges of $Q$, which has two connected components. This follows because whenever $d\sigma$ does not vanish on two opposite edges, then one can use the product of the extremal potentials for $\mathbb{P}^1$ with Poincar\'e type singularity at one fixed point and cone angle singularity at the other fixed point to give an extremal potential for $Q$. Hence the unstable weights for a parellogram are precisely the ones vanishing on opposite edges of $Q$.
\end{proof}

The method of proof of theorem \ref{mainquadthm} is as follows. We first show that given any weights, there is a \textit{formal ambitoric solution}, unless a simple condition necessary for stability is violated. A formal solution is a matrix-valued function $H^{ij}$ with the correct boundary conditions associated to $(Q,d\sigma)$ and for which $H^{ij}_{\textnormal{ } ij}$ is affine, but it may not be positive-definite everywhere in $Q^{\circ}$. We then show that stability is equivalent to the positive-definiteness of the formal solution. We also show that in this case $H^{ij}$ is in fact the inverse Hessian of a symplectic potential, so that in the case where $Q$ is Delzant this is equivalent to the existence of a genuine extremal metric on the corresponding toric orbifold.


\section{Formal solutions for quadrilaterals}\label{formalsols}

We begin this section by reviewing the construction of Apostolov-Calderbank-Gauduchon, which we will refer to as the ACG construction. It will suffice for us to describe the construction only briefly. In particular, we will omit a lot of the formulae that are not directly used. However these can be found in \cite[Sect. 3.2]{ACG15}. 

Given a quadrilateral $Q$ with no parallel edges, there is a $1$-parameter family of conics $\mathcal{C} (Q)$ such that the edges of $Q$ lie on tangent lines to $C (Q)$. Indeed, this condition just fixes four points on a dual conic $\mathcal{C}^* (Q)$, and there is a $1$-parameter family of conics going through these four points. Given such a conic, we can swipe out the quadrilateral $Q$ by taking the intersection of two tangent lines to $C(Q)$, provided we avoid having to use the tangent line to a point of $C(Q)$ at infinity.

Assuming this holds, we then get a new set of coordinates $(x,y)$ on $Q$, by parameterising $C(Q)$ and identifying a point in $(x,y) \in C(Q) \times C(Q)$ with the intersection of the tangent lines to $C(Q)$ at $x$ and $y$. The map is then well-defined away from the diagonal, and so to avoid any ambiguity we require $x > y$, so that $Q$ is the image under this map of a product of intervals $D = [\alpha_0, \alpha_{\infty}] \times [\beta_0, \beta_{\infty}]$ with 
\begin{align}\label{alphabetaincrease} \alpha_0 < \alpha_{\infty} < \beta_0 < \beta_{\infty}.
\end{align}
This will be \textit{positive ambitoric coordinates} for a quadrilateral $Q$.

Another way one could obtain new coordinates for a quadrilateral $Q$ is the following. Take a line $L$ with two marked points $p_1,p_2$. One can then parameterise all the lines going through $p_1$ and $p_2$, respectively, and take their intersections. This is well-defined provided we don't use the line $L$ itself. For a given quadrilateral $Q$, there are two pairs $(F_1,F_1')$ and $(F_2,F_2')$ of opposite sides of $Q$. These coordinates are then obtained by letting $p_i$ be the point corresponding to the intersection of $F_i$ and $F_i'$. We call these coordinates \textit{negative ambitoric coordinates}. This gives us a well-defined coordinate system provided the line containing $p_1$ and $p_2$ does not pass through the interior of the quadrilateral. Allowing one of the points $p_i$ to be at infinity gives trapezia, whereas allowing the line to be the line at infinity gives parallelograms. Again, we can assume this map is defined on some product $D = [\alpha_0, \alpha_{\infty}] \times [\beta_0, \beta_{\infty}]$ of closed intervals satisfying \ref{alphabetaincrease}.

Thus given the choice of such data, we get a map $\mu^{\pm}$, depending on whether we are considering positive or negative ambitoric coordinates. These send $D$ to quadrilaterals $Q^{\pm}$. For rational parameters, \cite{ACG15} showed that these were coordinates arising from what they call an ambitoric structure on a $4$-orbifold. However, the maps can also be seen as simply giving new coordinates for quadrilaterals.

\begin{rem}Any given quadrilateral can admit multiple ambitoric coordinate systems, depending on the choice of data above, and it can also admit both positive and negative ambitoric coordinates.
\end{rem}

We now fix ambitoric coordinates as above, either positive or negative, and let $A, B$ be \textit{quartic} polynomials such that
\begin{align}\label{boundcons1} A(\alpha_0) &= 0, A'(\alpha_0) = r_{\alpha_0},\nonumber \\
A(\alpha_{\infty}) &=0,  A'(\alpha_{\infty}) = r_{\alpha_{\infty}}, \nonumber \\
B(\beta_0) &= 0,  B'(\beta_{0}) = r_{\beta_{0}}, \\
B(\beta_{\infty}) &= 0,  B'(\beta_{\infty}) = r_{\beta_{\infty}}\nonumber,
\end{align}
and 
\begin{align}\label{boundcons2} A + B = q \pi.
\end{align}
Here the $r_{\gamma}$ are non-negative real numbers, $q (z) = q_0 z^2 + 2 q_1 z + q_2$ is a quadratic, positive on $[\alpha_0, \alpha_{\infty}] \times [\beta_0,\beta_{\infty}]$, which is fixed by the choice of ambitoric coordinates for $Q$, and $\pi$ is some other quadratic. This uniquely determines $A$ and $B$, as these are $10$ equations for $10$ unknowns. It was shown in \cite{ACG15} that these are in fact independent conditions.

Given $A,B$ satisfying the above, we can define $\mathfrak{t}$-invariant metrics on $D^{\circ} \times \mathfrak{t}$ by
\begin{align*} g_{\pm} =& \bigg( \frac{x-y}{q(x,y)} \bigg)^{\pm 1} \bigg( \frac{dx^2}{A(x)} + \frac{dy^2}{B(y)} + A(x) \big( \frac{y^2 d\tau_0 + 2 y d\tau_1 + d\tau_2}{(x-y) q(x,y)} \big)^2  + B(y) \big( \frac{x^2 d\tau_0 + 2 x d\tau_1 + d\tau_2}{(x-y) q(x,y)} \big)^2 \bigg) ,
\end{align*} 
provided $A,B$ are positive throughout $D^{\circ}$. Here $q(x,y)$ denotes $q_0 x y + q_1 (x+y) + q_2$ and $(\tau_0, \tau_1, \tau_2)$ are coordinates on the torus $\mathfrak{t}$ that satisfy 
\begin{align*} 2 q_1 \tau_1 = q_2 \tau_0 + q_0 \tau_2 .
\end{align*}

Regardless of whether or not $A$ and $B$ are positive, the projection of this to the $\mathfrak{t}$-fibres of the tangent bundle of $D^{\circ} \times \mathfrak{t}$ comes from a map $D^{\circ} \rightarrow S^2 \mathfrak{t}^*$, which moreover is actually the restriction of a smooth map $D \rightarrow S^2 \mathfrak{t}^*$.  

We can then use one of the maps $\mu^{\pm}$ to consider this as a map on $Q^{\pm}$ instead. From the formulae of \cite{ACG15}, the $\mu^{\pm}$ are defined on an open subset containing $D$, and so it takes smooth functions on $D$ to smooth functions on $Q^{\pm}$. Let $H_{\pm} : Q^{\pm} \rightarrow S^2 \mathfrak{t}^*$ be the function
\begin{align*} (x,y) \mapsto  \bigg( \frac{x-y}{q(x,y)} \bigg)^{\pm 1} \bigg( A(x) \big( \frac{y^2 d\tau_0 + 2 y d\tau_1 + d\tau_2}{(x-y) q(x,y)} \big)^2  + B(y) \big( \frac{x^2 d\tau_0 + 2 x d\tau_1 + d\tau_2}{(x-y) q(x,y)} \big)^2 \bigg)  .
\end{align*}
Then $H_{\pm}$ is smooth on $Q^{\pm}$. We then also have, as in \cite{ACG15}, that $H^{\pm}$ satisfies the boundary conditions required in \ref{intbypartslem} below for $Q^{\pm}$ with a boundary measure determined by the $r_k$ and a choice of lattice, which we take to be generated by the normals to two adjacent sides of $Q^{\pm}$. In \cite{ACG15}, it was also shown that $H^{ij}_{\textnormal{ } ij}$ is affine if and only if in equation \ref{boundcons2}, the quadratic $\pi$ is orthogonal to the quadratic $q$ under a suitable inner product.

Given a boundary measure $d \sigma$ on $\partial Q$, there is an associated affine function, see definition \ref{assafffn}. In this section we will follow \cite{ACG15} and call this affine function $\zeta$, as $A$ is used in the definition of an ambitoric metric above. We will need the following definition.
\begin{defn}[{\cite[Defn. 1.2]{elegendre11}}]\label{equipoisedefn} Let $Q$ be a quadrilateral with vertices $v_1, \cdots, v_4$ ordered such that $v_1$ and $v_3$ do not lie on a common edge of $Q$. An affine function $f$ on a quadrilateral $Q$ is \textnormal{equipoised} on $Q$ if
\begin{align*} \sum_i (-1)^i f(v_i) =0.
\end{align*}
A weighted quadrilateral $(Q,d\sigma)$ is an \textnormal{equipoised quadrilateral} if its associated affine function $\zeta$ is equipoised.
\end{defn}

There are many choices of ambitoric coordinates for a given quadrilateral. However, in the search for extremal potentials on weighted quadrilaterals, there is a preferred such coordinate system. In \cite{ACG15}, it was shown that almost all weighted rational Delzant quadrilaterals with rational weights admits ambitoric coordinates of the form above in which the solution $H^{ij}$ to the system \ref{boundcons1} has $\pi$ orthogonal to $q$, under a necessary condition for stability. However, their argument did not use the rationality of the weights nor of the quadrilateral and so holds in the setting where we consider irrational parameters, and non-negative boundary measures. 
\begin{lem}[{\cite[Lem. 4]{ACG15}}]\label{orthoglem} Let $(Q, d\sigma)$ be a weighted quadrilateral. Then provided $(Q,d\sigma)$ is not an equipoised trapezium, $Q$ admits ambitoric coordinates such that the matrix $H$ solving the system \ref{boundcons1} has $\pi$ is orthogonal to $ q $ if and only if $\phi(1,0)$ and $\phi(0,1)$ are positive.
\end{lem}
Here $\phi$ is the polynomial described in section \ref{genprops}. The points $(1,0)$ and $(0,1)$ correspond to the two simple piecewise linear functions with crease along a diagonal of $Q$.

We will call these coordinates \textit{preferred ambitoric coordinates for} $(Q,d\sigma)$, and to obtain extremal potentials from the ambitoric ansatz we necessarily have to work in these coordinates. For the case of equipoised trapezia, we will require a different construction of Calabi type toric metrics due to Legendre in \cite[Sect. 4]{elegendre11} that we describe in the next section.

The key in the argument of \cite{ACG15} to show that relative $K$-stability is equivalent to the existence of an ambitoric extremal metric, goes back to Legendre in \cite{elegendre11}, where she takes such an approach for positively weighted convex quadrilaterals which are equipoised. The idea is to use the formal solution $H^{ij}$ in preferred coordinates for $(Q,d\sigma)$, even though this is not necessarily positive-definite. One then shows that the positive-definiteness of $H^{ij}$ is equivalent to stability.

The crucial lemma for this argument in the case of positive boundary measure is a version of Donaldson's toric integration by parts formula in \cite{Don02}. The formula is applied to matrices that may not be the inverse Hessian of a function. In Donaldson's work, the $f$ are allowed to blow-up near the boundary at a certain rate. However, we will only need to consider smooth functions, so we only include these in our statement. In this case the proof is easier, as it is a direct application of Stokes's theorem, and so we omit it. This lemma has been used also in several other works such as in \cite{elegendre11}. The only difference is that we are allowing the $r_i$ to be $0$, which does not affect the proof.

\begin{lem}\label{intbypartslem} Let $P$ be a polytope in $\mathfrak{t}^*$, with facets $F_i = l_i^{-1} (0)$ for some affine functions $l_i$ that are non-negative on $P$. Let $u_i = d l_i$ be the conormal to $F_i$, and define a measure $d \sigma$ on $\partial P$ by $d \sigma_{| F_i} \wedge u_i= \pm d \lambda$, where $d\lambda$ is the Lebesgue measure on $\mathfrak{t}^*$. Suppose $H : P \rightarrow S^2 t^*$  is a smooth function on P such that on $\partial P$, 
\begin{align*} H(u_i , v) =&  0 \text{ for all } i \text{ and for all } v, \\
dH (u_i, u_i ) =&  r_i u_i \text{ for all } i,
\end{align*}
for non-negative numbers $r_i$. Then for any smooth function $f$ on $P$, 
\begin{align*} \int_P  H^{ij} f_{ij} d \lambda = \int_P H^{ij}_{\text{ }ij} f d\lambda + \int_{\partial P} f d \sigma_{\underline{r}},
\end{align*}
where $f_{ij}$ is the Hessian of $f$ computed with respect to a basis of $\mathfrak{t}^*$ whose volume form is $d \lambda$, $H^{ij}$ is the matrix obtained by evaluating $H$ on the dual basis for $\mathfrak{t}$ and $H^{ij}_{\text{ } kl}$ is the Hessian of the function $H^{ij}$ computed in these coordinates.
\end{lem}
The formal solutions from the preferred ambitoric coordinates will give functions satisfying these boundary conditions, and with $H^{ij}_{\text{ }ij}$ affine. We will then show that stability is equivalent to $H^{ij}$ being positive-definite. In the next section we will also show that if $H^{ij}$ is positive-definite, then it is the inverse of the Hessian of a symplectic potential.

We are now ready to prove that stability is equivalent to the existence of positive formal solutions. Since the ambitoric coordinates work equally well for non-Delzant quadrilaterals and for boundary measures that are arbitrary non-negative real numbers, the proof is exactly as in \cite{ACG15}. However, we include it for completeness. 

\begin{prop}\label{stabandformsoleq} Let $H_{A,B}$ be the formal extremal solution associated to a weight $d \sigma$ of a quadrilateral $Q$ admitting preferred ambitoric coordinates for this weight. Then $d\sigma$ is a stable weight if and only if $A,B$ are positive functions on $(\alpha_0, \alpha_{\infty})$ and $(\beta_0, \beta_{\infty})$, respectively. 
\end{prop}
\begin{proof} From \ref{intbypartslem} and that $H = H_{A,B}$ solves $H^{ij}_{\textnormal{ }ij} = \zeta$, it follows that 
\begin{align*} \mathcal{L} (f) = \int_P  H^{ij} f_{ij} d \lambda 
\end{align*}
for all smooth $f$. This can also be applied in the sense of distributions to piecewise linear functions, and one obtains as in \cite[p. 6]{ACG15}, that for simple piecewise linear functions with crease $I$,
\begin{align}\label{splACG} \mathcal{L} (f) = \int_{I} H (u_f, u_f) d\nu_f ,
\end{align}
where $u_f$ is a conormal to $I$ suitably scaled and $d\nu_f$ satisfies $u_f \wedge d \nu_f = d\lambda$. For a general piecewise linear function $f$, one gets a positive combination of such contributions over all creases of $f$. 

In particular, if $A,B$ are positive on the interior regions, then $H_{A,B}$ is positive-definite and so this is positive for all piecewise linear functions. Thus $(Q, d\sigma)$ is stable.

Conversely, suppose $A,B$ are not both positive on the interior regions. Assume first that $A(\alpha) \leq 0$ with $\alpha \in (\alpha_0, \alpha_{\infty})$. Then letting $f$ be a simple piecewise linear function with crease $I = \mu( \{ \alpha \} \times [\beta_0, \beta_{\infty}])$, one gets in \ref{splACG} that $H (u_f, u_f)$ is a positive multiple of $A(\alpha)$, and in particular $\mathcal{L} (f)$ is a positive multiple of $A(\alpha)$, and hence non-negative. Thus $(Q, d\sigma)$ is not stable. The argument for $B$ is identical, using a simple piecewise linear function with crease of the form $\mu( [\alpha_0, \alpha_{\infty}] \times \{ \beta \})$ instead.
\end{proof}


\section{The stable region}\label{stabinv}

In this section we will apply the ACG construction to arbitrary quadrilaterals with non-negative boundary measure to analyse the set of weights for which a quadrilateral is stable, and in particular prove theorem \ref{mainquadthm}.

We begin with a lemma giving a sufficient condition for a weighted quadrilateral to admit preferred ambitoric coordinates. Given two edges $E,F$, let $\phi,\psi$ be the functions $[0,1] \times [0,1] \rightarrow \mathbb{R}$ parameterising the Donaldson-Futaki invariant of simple piecewise linear functions with crease meeting the two edges adjacent to $E$ and $F$, respectively. We can suppose $(0,0)$ is the vertex of $[0,1] \times [0,1]$ corresponding to the affine function vanishing exactly along $E$ and similarly for $\psi$ and $F$. Then $(1,0)$ and $(0,1)$ correspond to the two simple piecewise linear functions with crease a diagonal of $Q$, both for $\psi$ and $\phi$. 

\begin{lem}\label{prefcoordslem} Let $(Q,d\sigma)$ be a weighted quadrilateral with $d\sigma$ vanishing on two edges $E$ and $F$. If the Hessians of $\phi$ and $\psi$ at $(0,0)$ are both positive semi-definite, then $\phi$ and $\psi$ are positive at $(1,0)$ and $(0,1)$. 
\end{lem}

\begin{proof} The proof uses direct computation. Consider the one-parameter family of boundary measures $d\sigma_r$ as in the statement of \ref{mainquadthm}, and so we have corresponding polynomials $\phi_r$ and $\psi_r$. Note that $\phi_r(1,0)$ is linear in $r$, and similarly for $\phi_r (0,1)$. Let $r_1, r_2$ be the values for which $\phi_r(1,0)=0$ and $\phi_r(0,1)=0$, respectively. 

A calculation shows the key property for our purposes, namely that the sign of the determinant of the Hessian of $\phi_{r_i}$ at $(0,0)$ is the \textit{opposite} of the sign of the determinant of the Hessian of $\psi_{r_i}$ at $(0,0)$. Thus the set of $r$ for which these determinants are both positive is contained in the region where $\phi_r(1,0)$ and $\phi_r(0,1)$ have the same sign. Moreover, when $r=0,1$ at most one of the diagonals can correspond to a destabilising simple piecewise linear function. In particular, the region in which $\phi_r(1,0)$ and $\phi_r(0,1)$ have the same sign must intersect $[0,1]$ and necessarily be such that this sign is positive. Then the region where the determinant condition holds must be contained in this region and the result follows.
\end{proof}

We now use proposition \ref{stabandformsoleq} to give an easily computable criterion for stability on weighted quadrilaterals with no parallel sides and with $0$ boundary measure on two adjacent sides. For a quadrilateral $Q$ with edges $E_1, \cdots, E_4$, let $\phi$ and $\psi$ be the two functions corresponding to evaluating $\mathcal{L}$ on simple piecewise linear functions with crease meeting two opposite edges of $Q$. Also, given edges $E_i, E_j$, let $d\sigma_r$ be the measure corresponding to the formal sum $r E_i + (1-r) E_j$, as in the statement of theorem \ref{mainquadthm}.

\begin{prop}\label{adjsidesprop} Let $E_i, E_j$ be adjacent sides of $Q$. Then $(Q, d\sigma_r)$ for $r \neq 0,1$ is stable if and only if the Hessians of the functions $\phi$ and $\psi$ are positive semi-definite at the points corresponding to the SPL function whose crease is an edge with $0$ boundary measure.
\end{prop}
\begin{proof} First note that under these conditions $(Q,d\sigma)$ is never an equipoised trapezium. Now, if the Hessians of $\phi$ and $\psi$ at the points $p,q$ corresponding to the simple piecewise linear function with crease an edge with $0$ boundary are not positive semi-definite, then $(Q, d\sigma_r)$ is not stable. Indeed, from lemma \ref{prefcoordslem}, $p$ and $q$ are critical points of $\phi$ and $\psi$. Thus if the positive semi-definiteness does not hold, then either $\phi$ or $\psi$ decreases in some ray away from $p$ or $q$. Since $\phi$ and $\psi$ are $0$ at $p$ and $q$, respectively, it follows that $\mathcal{L}_{\underline{r}}$ is negative on some simple piecewise linear function, hence $\underline{r}$ is not a stable weight for $Q$.

Conversely, suppose the Hessians are positive semi-definite. From lemma \ref{prefcoordslem}, $(Q, d\sigma_r)$ admits preferred ambitoric coordinates. So we must show that the formal solution $H_{A,B}$ has $A$ and $B$ positive in $(\alpha_0, \alpha_{\infty})$ and $(\beta_0, \beta_{\infty})$, respectively. 

We first show that $A$ is positive if the Hessian of $\phi(s,t)$ is positive semi-definite at $p$, where $\phi$ is the Donaldson-Futaki invariant of simple piecewise linear functions with crease along the edges corresponding to $y = \beta_0$ and $y=\alpha_{\infty}$. Consider the Donaldson-Futaki invariant of functions $f_c$ with crease $x = c$. From the proof of \ref{stabandformsoleq} we have that this is given by 
\begin{align*} \mathcal{L} (f_c) = A(c) h_c,
\end{align*}
where $h_c$ is a function obtained from integrating a smooth positive function over $I_{c} = \{ (c, t) : t \in [\beta_0, \beta_{\infty}]\}$. In particular, if for simplicity the edge with $0$ boundary measure is $x = \alpha_0$, we have that $A(\alpha_0) = A'(\alpha_0) = 0$, and so
\begin{align*} \frac{d^2}{ds^2}_{|s=0} ( \mathcal{L} (f_{\alpha_0 + s}) ) = A'' (\alpha_0) h_{\alpha_0}.
\end{align*} 
It follows that $A''(\alpha_0) \geq 0$, since $\phi_{s,t}$ is positive semi-definite. 

Now, $A$ is a polynomial of degree $4$ with a double zero at $\alpha_0$ and a simple zero at $\alpha_{\infty}$. Moreover, the condition on $A'(\alpha_{\infty})$ implies that $A$ is positive near $\alpha_{\infty}$. If $A''( \alpha_0) > 0$, then $A$ is positive near $\alpha_0$, too, and so this means that $A$ must have two more zeros, counted with multiplicites, if $A$ is not positive in $(\alpha_0, \alpha_{\infty})$. But this means that $A$ has five zeros, counted with multiplicities, and so $A$ has degree at least $5$, a contradiction. If $A''(\alpha_0) = 0$, then $A$ can have no zeros in $(\alpha_0, \alpha_{\infty})$, since it has degree $4$ and we have $4$ zeros at $\alpha_0$ and $\alpha_{\infty}$ counted with multiplicity. In particular, $A$ has constant sign in this interval. Since $A$ is positive near $\alpha_{\infty}$, it therefore follows that $A$ is positive in $(\alpha_0, \alpha_{\infty})$. 

Similarly, one obtains the result for the case when the $0$ boundary measure occurs at $\alpha_{\infty}$. The same argument also works to show that $B$ is positive in $(\beta_0, \beta_{\infty})$ if $\psi$ has positive semi-definite Hessian at $q$.
\end{proof}

We are now ready to prove an analogous result for the case when opposite sides have $0$ boundary measure. We analogously get a criterion that is easy to compute, but note that in this case it is an \textit{open} condition. Note also that we only need to check this for \textit{one} of the functions $\phi,\psi$. Lemma \ref{oppsideslem} below, which forms part of the proof, will show that if $\phi$ is positive-definite at $(0,0)$, then $\psi$ is automatically positive.

\begin{prop}\label{oppsidesprop} Let $(Q,d\sigma)$ be a weighted quadrilateral where $d\sigma$ vanishes exactly on two opposite edges and such that $(Q,d\sigma)$ is not an equipoised trapezium. Then $(Q, d\sigma)$ is stable if and only if the Hessian of the function $\phi$ is positive-definite at the point $(0,0)$.
\end{prop}

\begin{rem} We will see below that in the case of equipoised trapezia the same conclusion holds, but for now we will consider the cases where we can apply the ACG construction.
\end{rem}

\begin{proof} As before, the points corresponding to an affine function are critical points of $\phi$. Therefore, if the Hessian is not positive semi-definite at these points, then $\phi$ decreases in some direction. Since $\phi$ is zero at these points, it follows that if the Hessian is not positive semi-definite, then $(Q, d\sigma)$ is not stable. 

Now assume the determinant condition holds. By lemma \ref{prefcoordslem}, $(Q, d\sigma)$ admits preferred ambitoric coordinates. Without loss of generality, assume that in these coordinates, the edges with $0$ boundary measure correspond to $x = \alpha_0$ and $x= \alpha_{\infty}$, respectively. Let $A,B$ be the quartics for the formal solution $H_{A,B}$. 

To show that the $B$ is positive, it suffices to show that all $B$'s in a region of weights containing the weights for which the Hessian condition holds, are positive. This follows from lemma \ref{oppsideslem} below and the positivity of the $A$ on $(\alpha_0, \alpha_{\infty})$ for all weights satisfying the Hessian condition that we will now show. 

We do this by a similar argument as in the adjacent case, considering the second derivative of the Donaldson-Futaki invariant 
\begin{align*} \mathcal{L} (f_c) = A(c) h_c
\end{align*}
of the family $f_c$ of simple piecewise linear functions with crease $x=c$. The sign of this is the same as $A''(\alpha_0)$. Now, since $A$ is a quartic with a double zero at both $\alpha_0$ and $\alpha_{\infty}$, it follows that $A(z) = \lambda (z-\alpha_0)^2 (z-\alpha_{\infty})^2$. Thus 
\begin{align*} A''(\alpha_0) = \lambda (\alpha_0 - \alpha_{\infty})^2.
\end{align*}
In particular, the sign of $\lambda$, which is positive if and only if $A$ is positive on $(\alpha_0, \alpha_{\infty})$, equals the sign of the second derivative of $\mathcal{L} (f_c) $. Since $\phi$ is positive-definite at the critical point, it follows that $\mathcal{L} (f_c) >0$. Thus $\lambda > 0$ and so $A$ is positive throughout $(\alpha_0, \alpha_{\infty})$. 

Finally we must consider the borderline case when the Hessian is strictly positive semi-definite. The above also shows that if $\phi$ is only positive semi-definite, then $\mathcal{L}$ vanishes on functions with crease $x= \alpha$ for all $\alpha \in [\alpha_0, \alpha_{\infty}]$. So in this case positive semi-definiteness is not sufficient, one needs $\phi$ to be positive-definite.
\end{proof}

To complete the proof of proposition \ref{oppsidesprop}, we must show the following lemma.

\begin{lem}\label{oppsideslem} Let $(Q, d\sigma)$ be a quadrilateral admitting preferred ambitoric coordinates and for which $d\sigma$ vanishes on two opposite edges of $Q$. Moreover, suppose $d\sigma$ is such that the Hessian of $\phi$ is strictly positive semi-definite at $(0,0)$. Then for the formal solution $H_{A,B}$ associated to $d\sigma$, $B$ is positive on $(\beta_0, \beta_{\infty})$.
\end{lem}
\begin{proof} Under this Hessian condition, it follows that in the formal solution $A$ is identically zero. Thus the formal solution satisfies
\begin{align*} \pi(z) q(z) &= A(z) + B(z) \\
&= B(z).
\end{align*}
Thus $q$ divides $B$ and so the zeros of $B$ are $\beta_0, \beta_{\infty}$ and the zeros of $q$. However, recall that $q$ must be chosen so that it does not have any zeros in $[\beta_0, \beta_{\infty}]$. Thus $B$ has no zeros in $(\beta_0, \beta_{\infty})$, and so has constant sign in this interval. Since the boundary conditions imply that $B$ increases from $\beta_0$, it therefore follows that $B$ is positive throughout $(\beta_0, \beta_{\infty})$, as required.
\end{proof}

We have now found an easily computable criterion for stability for all weighted quadrilaterals with boundary measure vanishing on two edges, apart from equipoised trapezia, where either the boundary measure vanishes on two non-parallel sides or the boundary measure vanishes on the two parallel sides and is equal on the two non-parallel sides, using a normalisation as in \cite[Eqn. 4.7]{elegendre11}. 

Of these two cases, the former are always unstable by an example due to Sz\'ekelyhidi.

\begin{prop}[{\cite[Prop. 15]{szekelyhidi08}} ] Suppose $Q$ has parallel sides. Then for any boundary measure which is supported on the two parallel sides, $(Q, d\sigma)$ is strictly semistable. 
\end{prop}

For the latter case, let $E$ and $F$ be the two sides that are not parallel. We can apply a simple argument using the ACG construction in almost all situations to determine the stability of the boundary measure $d\sigma_r$ corresponding to $r E + (1-r) F$. The construction applies to all but one value of $r$, say $r'$. Doing this we get from proposition \ref{oppsidesprop} that for $r \in (0,1) \setminus \{r'\}$, $d\sigma_r$ is a stable weight for all $r \in (r_0,r_1) \setminus \{ r' \}$, for some $r_0,r_1$. Since the stable set is connected it follows that provided $r'$ is neither $r_0$ nor $r_1$, the stable set is $(r_0,r_1)$. 

To rule out that $r'$ can be one of the $r_i$ and be a stable weight, we briefly mention the construction of Legendre in \cite{elegendre11} for equipoised trapezia. From this it will also be clear that the arguments of the next section will apply to this construction, so that stability for such trapezia are equivalent to the existence of an extremal metric with Poincar\'e type singularities along the two divisors corresponding to the opposite edges.

In this case, one can realize the moment polytope as the image of $[\alpha_1, \alpha_2] \times [\beta_1, \beta_2]$ under the map
\begin{align*} (x,y) \mapsto (x, xy),
\end{align*}
for some $\alpha_2 > \alpha_1 >0$ and $\beta_2 > \beta_1 \geq 0$. Let $t_1,t_2$ be the angle coordinates corresponding to these coordinates. One obtains metrics from two functions $A : [\alpha_1, \alpha_2] \rightarrow \mathbb{R}$ and $B : [\beta_1, \beta_2] \rightarrow \mathbb{R}$, positive on the interiors of their domains, as
\begin{align*} \frac{x dx^2}{A(x)} + \frac{x dy^2}{B(y)} + \frac{A(x)}{x} (dt_1 + y dt_2)^2 + x B(y) dt_2^2,
\end{align*}
whenever $A,B$ vanish at the end-points, and the derivatives of $A$ and $B$ at the end-points are determined by $d\sigma$, for positive weights. 

The extremal condition is that $A$ is a polynomial of degree at most $4$, $B$ is a polynomial of degree $2$ with leading term $- a_2$, where $a_2$ is the coefficient of $x^2$ in $A$. This determines $A$ and $B$ uniquely and puts a condition on the conormals, only involving those along the edges $y=\beta_1$ and $y=\beta_2$. As before we can let $A$ have double zeros at either end-points, which correspond to the boundary measure vanishing at $x= \alpha_1$ or $x= \alpha_2$. In the case when the boundary measure vanishes on both sides, one can then use exactly the same arguments as before to determine that stability is equivalent to this formal solution being positive on the interior, and that this in turn is equivalent to the Hessian condition of \ref{oppsidesprop}.

We now analyse the stable region, with the goal of proving that the region in which the Hessian condition is satisfied is non-empty, unless $(Q, d\sigma)$ satisfies the conditions of Sz\'ekelyhidi's example. We first consider measures supported on adjacent sides.

\begin{prop}\label{adjsidesstab} Let $Q$ be a quadrilateral and fix two adjacent sides $E$ and $F$. Then there exists a boundary measure $d\sigma$ supported on $E$ and $F$ such that $(Q, d\sigma)$ is stable.
\end{prop}

The key step to proving this is the following lemma.
\begin{lem}\label{subtrilem} Let $(Q,d\sigma)$ be a weighted quadrilateral, and let $\Delta_1, \Delta_2$ be the two triangles obtained by splitting $Q$ in two via a diagonal. Define a boundary measure $d \tau_i$ on $\Delta_i$ to be $d\sigma$ on the edges shared with $Q$ and $0$ on the edge corresponding to the diagonal of $Q$. Then the associated affine linear functions $A_i$ of $(\Delta_i, d\tau_i)$ are never equal.
\end{lem}

\begin{proof} We can assume that $\Delta_1$ has vertices $(-1,0)$, $(0,1)$ and $(0,c)$ with $c > 0$ and that $\Delta_2$ as vertices $(-1,0),(0,1)$ and $(p,-q)$ with $q > 0$. The edges $E_i$ of $Q$ have defining functions
\begin{align*} l_1 (x,y) &= c- cx - y , \\
l_2 (x,y) &= c + c x -y , \\
l_3 (x,y) &= q + qx +(1+p)y, \\
l_4(x,y) &= q-qx+(1-p)y.
\end{align*}
We must also assume that $l_1 (p,q)$ and $l_2(p,q)$ are positive to ensure that $\Delta_1 \cup \Delta_2$ is a convex quadrilateral. The boundary measure along $E_i$ can be written as $r_i dy$, for some $r_i \geq 0$, but not all $0$. Let $d\sigma_1$ be the boundary measure for $\Delta_1$ and similarly for $\Delta_2$. 

A long but elementary calculation shows that the affine linear function $A_i$ associated to $(\Delta_i, d\sigma_i)$ is given by
\begin{align*} A_1 (x,y) )=& 3 (r_1 - r_2) x + \frac{3(r_1 + r_2)}{c} y, \\
A_2 (x,y) =& 3 (r_4 - r_3) x - \frac{3(r_3 + r_4 + r_3 p - r_4 p)}{q} y.
\end{align*}
We must show that they never can be equal provided $Q$ is convex.

First of all, if $p \in (-1,1)$, then the coefficient of $y$ for $A_2$ is negative. Since the coefficient of $y$ for $A_1$ is always non-negative, this means we must have $p \notin (-1,1)$. By symmetry it suffices to check all the cases where $p\geq1$, so we need to check that there is no solution for $p \in [1,1+\frac{q}{c}]$, the end-point $1+\frac{q}{c}$ coming from the condition that $Q$ is convex. 

For this one can check that the general solution to $A_1 = A_2$ giving $r_3$ and $r_4$ in terms of $r_1$ and $r_2$ is affine linear in $p$. In particular, if $r_3$ is negative for $p =1$ and for $p=1+ \frac{q}{c}$ whenever $r_1$ and $r_2$ are positive, then this holds for all $p \in  [1,1+\frac{q}{c}]$ and we are done. 

This is indeed the case as one can check that at $p =1$, the solution is
\begin{align*} r_3 = - \frac{q(r_1+r_2)}{c},
\end{align*}
and at $p=1 + \frac{q}{c}$, the solution is
\begin{align*} r_3 = -\frac{r_2 q}{c},
\end{align*}
both of which are negative.
\end{proof}

We can now prove proposition \ref{adjsidesstab}.
\begin{proof} If there were no such weights, then there would have to exist a strictly semistable $(	Q,d\sigma)$ with unique destabilising simple piecewise linear function given by a diagonal of $Q$. Indeed, there would certainly have to be one such boundary with crease going through one vertex of $Q$. If this was the case and this was not a diagonal, then $(Q,d\sigma)$ would admit preferred ambitoric coordinates. But then the simple piecewise linear functions with crease the ambitoric coordinate lines would have positive Donaldson-Futaki invariant. In particular, the formal solution $H_{A,B}$ would be positive-definite and so $(Q, d\sigma)$ would be stable, a contradiction.

If there was such a strictly semistable polytope whose unique destabilising function had crease a diagonal, it would follow from an argument similar to one given in \cite{Don02}, that the corresponding weighted subpolytopes $(\Delta_i, d \tau_i)$ would then have equal associated affine linear function $A_i$. But this violates lemma \ref{subtrilem}.
\end{proof}

Finally, we consider the case when the boundary measure is supported on opposite sides.

\begin{prop}\label{oppsidesstab} Let $Q$ be a quadrilateral and fix two opposite sides $E$ and $F$. Then there exists a boundary measure $d\sigma$ supported on $E$ and $F$ such that $(Q, d\sigma)$ is stable if and only if $E$ and $F$ are not parallel.
\end{prop}

We already know one direction of this proposition due to Sz\'ekelyhidi's example, so to prove \ref{oppsidesstab}, we thus have to show that if $E$ and $F$ are opposite sides that are not parallel, then there exists a stable weight. However, it will also be transparent in the proof that both directions are true. 

\begin{proof} We use the determinant condition of proposition \ref{adjsidesprop}. This determinant is a quadratic in $r$. One can show that at the critical point of this quadratic, the value is
\begin{align*} \frac {{p}^{2}{k}^{4}q \left( kp+k+q \right) ^{2}}{4 \rho_1 \rho_2},
\end{align*}
where
\begin{align*} \rho_1 &= {k}^{2}{p}^{2}+2\,{k}^{2}p+2\,kpq+{k}^{2}+2\,kp+2\,kq+{q}^{2}+2\,k, \\
\rho_2 &=  {k}^{2}{p}^{2}+2\,{k}^{2}pq+2\,{k}^{2}p+2\,kpq+2\,k{q}^{2}+{k}^{2}+2\,kq+{q}^{2} .
\end{align*}
Here we are using the formulae given for the quadrilateral $Q$ and the boundary measure as in section \ref{genprops}.

The numerator of this is always positive unless $p=0$, which is the case when $E$ and $F$ are parallel. Note that $kp + k + q \neq 0$. It is in fact positive, since $p>-1$ and $q>0$. In the case of the denominator, we consider each factor separately. These are both quadratics in $p$, so it suffices to show that there are no zeros of these quadratics for the allowed values of $p$, and that at some point they are positive.

For $\rho_1$ the roots are
\begin{align*} p &={\frac {-k-q-1 \pm \sqrt {2\,q+1}}{k}}.
\end{align*}
Since $k>0$, the larger of these roots is the one taking the positive sign, and so we must show that such a root is smaller than either $-1$ or $-\frac{q}{k}$. But if
\begin{align*} {\frac {-k-q-1+\sqrt {2\,q+1}}{k}} &>-1,
\end{align*}
then  
\begin{align*}\sqrt{2q +1} > q + 1.
\end{align*}
Since both sides are greater than $0$, this inequality is preserved when squaring, and so this implies 
\begin{align*} q^2 < 0,
\end{align*}
a contradiction. Thus all roots of the first factor satisfy that $p<-1$, hence it is positive for any convex quadrilateral.

For $\rho_2$, the roots are
\begin{align*} p &= \frac {-kq-k-q \pm \sqrt {{k}^{2}{q}^{2}+2\,{k}^{2}q}}{k}.
\end{align*}
The greater of these is again taking the positive sign, and if 
\begin{align*} \frac {-kq-k-q + \sqrt {{k}^{2}{q}^{2}+2\,{k}^{2}q}}{k} > - \frac{q}{k},
\end{align*}
then 
\begin{align*} \sqrt {{k}^{2}{q}^{2}+2\,{k}^{2}q} > kq + k.
\end{align*}
Squaring, this would imply that
\begin{align*} k^2 < 0,
\end{align*}
again a contradition. Thus both terms are positive whenever $p> \textnormal{max} \{ - \frac{q}{k}, -1 \}$.

We have shown that the critical weight is a stable weight unless $E$ and $F$ are parallel. What remains is to show that the critical weight is a valid weight, i.e. lies in $(0,1)$. Since the determinant condition is violated at both $r=0$ and $r=1$, it therefore suffices to show that the determinant increases at $r=0$ to conclude that the critical $r$ must lie in $(0,1)$. 

A computation shows that the derivative of the determinant at $r=0$ being positive is equivalent to 
\begin{align}\label{positivitycheck}& \nonumber{k}^{4}{p}^{4}+4\,{k}^{4}{p}^{3}+4\,{k}^{3}{p}^{3}q+6\,{k}^{4}{p}^{2}+4\,{k}^{3}{p}^{3}+16\,{k}^{3}{p}^{2}q+6\,{k}^{2}{p}^{2}{q}^{2}+4\,{k}^{4}p \\
&+12\,{k}^{3}{p}^{2} +16\,{k}^{3}pq+8\,{k}^{2}{p}^{2}q+16\,{k}^{2}p{q}^{2}+4\,kp{q}^{3}+{k}^{4}+12\,{k}^{3}p+4\,{k}^{3}q\\
&\nonumber+16\,{k}^{2}pq +10\,{k}^{2}{q}^{2}+4\,kp{q}^{2}+4\,k{q}^{3}+{q}^{4}+4\,{k}^{3}+8\,{k}^{2}q+4\,k{q}^{2}
\end{align}
being positive. Now, if $q\geq k$, then $p > -1$. In this case, we make the substitution $p = -1 + a$ above, so $a>0$. We then have that \ref{positivitycheck} becomes
\begin{align*}{a}^{4}{k}^{4}+4\,{a}^{3}{k}^{3}q+4\,{a}^{3}{k}^{3}+4\,{a}^{2}{k}^{3}q+6\,{a}^{2}{k}^{2}{q}^{2}+8\,{a}^{2}{k}^{2}q-4\,{k}^{3}qa+4\,{k}^{2}{q}^{2}a+4\,k{q}^{3}a+4\,k{q}^{2}a+{q}^{4}.
\end{align*}
Since $a,k$ and $q$ are positive, the only negative term above is $-4 k^3qa$. However, since we are assuming $q \geq k$, this is dominated by the term $+ 4kq^3a$. Hence this is always positive for all $a,k,q>0$.

In the case when $q \leq k$, one can use the substitution $p = -\frac{q}{k} + a$ instead and use a similar argument to obtain the same conclusion. Thus the derivative of the determinant is positive at $r=0$, and this completes the proof.
\end{proof}

From the above results we thus get the following characterisation of the stable weights for a quadrilateral, which is simply theorem \ref{mainquadthm} with the numbers in the statement explicitly given. Recall from section \ref{genprops} that fixing two edges $E_i$ and $E_j$ we have two associated polynomials $\phi$ and $\psi$.

\begin{cor}\label{detcondcor} Let $Q$ be a quadrilateral and fix two edges $E_i, E_j$ of $Q$. Let $c_0,c_1$ denote weights $(1-r)E_i + rE_j$ for which the determinant of $\phi$ vanishes at $(0,0)$ and similarly define $c_2,c_3$ for $\psi$. Then the weights of this form which are stable weights for $Q$ are precisely given by
\begin{itemize} \item the intersection of $[c_0,c_1]$ and $[c_2,c_3]$ with $(0,1)$ if $E_i$ and $E_j$ are adjacent,
\item $(c_0,c_1) \cap (0,1)$ if $E_i$ and $E_j$ are opposite.
\end{itemize}
This is always non-empty unless $E_i$ and $E_j$ are parallel edges, and in this case all such weights are unstable.
\end{cor}

It follows from the examples of unstable positive weights for quadrilaterals with no parallel sides in \cite[Prop. 6]{ACG15} that for a quadrilateral with no parallel sides, and with boundary measure $d\sigma$ supported on one edge only, $(Q,d\sigma)$ is unstable and not strictly semistable, i.e. $\mathcal{L} (f) < 0$ for some simple piecewise linear function $f$.  From this and our characterisation of the stable set along edges of the simplex $\sum_i r_i = 1$, it follows that given two adjacent edges $E$ and $F$ on such a quadrilateral, the $r$ such that $r E + (1-r) F$ is stable is a \textit{closed} non-empty interval, contained in $(0,1)$. Thus we get the following corollary.

\begin{cor} Let $Q$ be a quadrilateral without parallel sides. Then the set of weights $d\sigma$ for which $(Q, d\sigma)$ is stable, identified with a subset of $\mathbb{R}^{4}_{\geq 0} \setminus \{ 0 \}$, is neither open nor closed.
\end{cor}

This is surprising. When looking at positive weights the set of weights in $\mathbb{R}^4_{>0}$ for which $(Q, d\sigma)$ is stable \textit{is} open. Indeed, in Donaldson's continuity method for extremal metrics on toric surfaces in \cite{donaldson08}, he in particular showed that for any polytope with positive weights, the set of weights which admits an extremal potential is open. The openness of the stable set then follows as we will show in the next section that when $d \sigma$ is positive on each edge of $Q$ and $(Q, d\sigma)$ is stable, then the formal solution is the inverse Hessian of a symplectic potential. In the next section we will also discuss how the points on the boundary of the stable region can be explained by the formation of metrics with different sort of asymptotics than the Poincar\'e type metrics.

\begin{rem} If the optimal destabiliser of an unstable polytope is a simple piecewise linear function, then the polytope splits into two semistable pieces, and these pieces have vanishing boundary measure on only one edge. It is therefore interesting to see if this sort of behaviour can happen when one allows only one edge with zero boundary measure. The answer to this is \textnormal{yes}. This follows from the above and the convexity of the stable set. 

Pick an edge $E$ of $Q$ and let the two edges adjacent to it be $F_1, F_2$. Then there exists minimal $r_1,r_2$ such that $d\sigma_1 = (1-r_1) E + r_1 F_1$ and $d\sigma_2 = (1-r_2) E + r_2 F_2$ are stable weights, respectively. It follows that all convex combinations $(1-r) d\sigma_1 + r d\sigma_2$ of these two weights are stable, and these vanish only on the edge opposite $E$ when $r \in (0,1)$.
\end{rem}


\section{Relation to the existence of extremal metrics}\label{relstometssection}

In the previous section, we showed how the existence of a positive-definite formal solution is equivalent to stability. In this section we will show that if the formal solutions are positive-definite then the $H^{ij}$ is the inverse Hessian of a symplectic potential. These  (generically) correspond to Poincar\'e type metrics in a weak sense on the edges with weight $0$, however we show that a different behaviour occurs too. This will explain the non-openness of the stable set in the previous section. This corresponds to symplectic potential having the behaviour $u^{11} = O(x^3)$ near an edge lying in $x =0$. This in turn correspond to the metric being modelled on 
\begin{align*} \frac{|dz|^2}{|z|^2 (- \log(|z|) )^{\frac{3}{2}}}
\end{align*}
near the divisor corresponding to this edge. 

The section has three parts. First we consider edges where the boundary measure is positive, which is the case consider in \cite{ACG15}. Next, we take the case when the $A,B$ have exactly double roots on edges where the boundary measure vanishes, and finally we consider the case of triple root. We emphasise that in this section we consider all metrics coming from the ACG construction. In other words, $A,B$ can be arbitrary positive functions on $(\alpha_0,\alpha_{\infty})$ and $(\beta_0, \beta_{\infty})$, respectively, satisfying the boundary conditions \ref{boundcons1} and \ref{boundcons2}, not necessarily extremal potentials.

At edges with non-vanishing boundary measure, we get metrics with cone angles determined by $d\sigma$. Below we let $F$ be the union of the edges where $d\sigma$ vanishes.

\begin{lem}\label{ACGconeanglelem} Let $(Q, d\sigma)$ be a weighted quadrilateral admitting positive or negative ambitoric coordinates defined on $[\alpha_0, \alpha_{\infty}] \times [\beta_0 , \beta_{\infty}]$. Let $A$ and $B$ be positive on $(\alpha_0, \alpha_{\infty})$ and $(\beta_0, \beta_{\infty})$, respectively, satisfying \ref{boundcons1} and \ref{boundcons2}, but which do not have to be quartics. Let $H_{A,B}$ be the corresponding positive-definite map $Q \rightarrow S^2 \mathfrak{t}^*$. Then $H_{A,B}$ is the inverse Hessian of a symplectic potential on $Q^{\circ}$ which satisfies the Guillemin boundary conditions at each edge with non-zero boundary measure. In particular, the $H_{A,B}$ equals the inverse Hessian of a function $u$ which can be written as 
\begin{align*} \frac{1}{2} \sum_{\{i : d\sigma_{| E_i} \neq 0 \} } l_i \log l_i + h,
\end{align*}
where $h \in C^{\infty} (P^{\circ}) \cap C^{0} ( P \setminus  F )$ and $l_i$ is the affine linear function defining $l_i$ determined by $d\sigma$.
\end{lem}

\begin{proof} The proof follows from the analogous statement for rational weights, proved in \cite{ACG15}. Indeed, by changing basis by a transformation which is not necessarily in  $\textnormal{SL}_{2}(\mathbb{Z})$, one gets that $(Q, d\sigma)$ gets mapped to a quadrilateral where the boundary measure along our given edge $E$ is the standard one. Thus we get that the open polytope $Q^{\circ} \cup E^{\circ}$ and the composition of the previous ambitoric coordinates with these transformations are ambitoric coordinates for this polytope with the new weight. Since this is rational data, it follows that it comes from an ambitoric structure on $\mathbb{C} \times \mathbb{C}^*$, and in particular by a symplectic potential satisfying the standard boundary conditions along $E$. It therefore follows that the original positive-definite matrix $H_{A,B}$ also comes from a symplectic potential satisfying the Guillemin boundary conditions along $E$ determined by $d\sigma$.
\end{proof}

Next, we consider the edges with $0$ boundary measure, where the corresponding function vanishes exactly to second order. 

\begin{prop}\label{ptacgmets} Let $H_{A,B}$ be the function associated to an ambitoric structure on a weighted quadrilateral $(Q, d\sigma)$ as in lemma \ref{ACGconeanglelem}. Suppose $A,B$ vanish exactly to second order at the points corresponding to the edges with $0$ boundary measure. Then $ u $ is quasi-isometric to a metric induced by a symplectic potential in $ \mathcal{S}_{Q, d\sigma}$ to any order.
\end{prop}

\begin{proof} We do the proof in the case of positive ambitoric coordinates. The proof in the negative case is similar.

Let $x = \alpha_0$ be an edge with $0$ boundary measure. Let the symplectic coordinates be $\chi$ and $\eta$, which turn out to be given by
\begin{align*} \chi &= \frac{(x- \alpha_0)(y - \alpha_0)}{q(x,y)} ,\\
\eta &= \frac{(\beta_0 - x)(y - \beta_0)}{q(x,y)} .
\end{align*}
 We then have that, for example
\begin{align*} \frac{\partial \chi}{\partial x} = \frac{(y-\alpha_0)(q(x,y) - (x-\alpha_0) \frac{\partial q}{\partial x} (x,y) )}{q^2(x,y)}.
\end{align*}
Recall the inequalities \ref{alphabetaincrease}, so that e.g. $y- \alpha_0$ is positive and bounded away from zero. Since $q(x,y)$ is smooth and positive in a neighbourhood of $[\alpha_0, \alpha_{\infty}] \times [\beta_0, \beta_{\infty}]$, it follows from this and similar calculations for the other entries in the Jacobian of this coordinate change that taking derivatives with respect to the $(x,y)$ and $(\chi,\eta)$ variables are mutually bounded.

Taking $n$ derivatives of the Hessian $u_{ij}$ of the symplectic potential $u$ in the $\chi$ direction is therefore mutually bounded with taking $n$ derivatives in the $x$-direction of 
\begin{align*} \frac{q(x,y)}{(x-y)A(x)} .
\end{align*}
This is in turn mutually bounded with taking $n$ derivatives of 
\begin{align*} \frac{1}{(x-\alpha_0)^2}
\end{align*}
with respect to $x$, since $A$ vanishes exactly to order $2$ at $\alpha_0$. Hence it is mutually bounded with
\begin{align*} \frac{1}{(x-\alpha_0)^{2+n}}.
\end{align*}
Taking $n$ derivatives of the Hessian of the model symplectic potential in the $\chi$-direction is mutually bounded with $\frac{1}{\chi^{2+n}}$, which in turn near $x=\alpha_0$ is mutually bounded with 
\begin{align*} \frac{1}{(x-\alpha_0)^{2+n}}
\end{align*}
as well. Thus the symplectic potential $u$ is mutually bounded with the model for derivatives to any order. 
\end{proof}

\begin{prop} Let $H_{A,B}$ be the function associated to an ambitoric structure on a weighted quadrilateral $(Q, d\sigma)$ as in lemma \ref{ACGconeanglelem}. Suppose $A$ or $B$ vanish to third order at a point corresponding to an edge $E = l^{-1} (0)$ with $0$ boundary measure. Then $u$ has the asymptotics of the model potential where one exchanges the term $- a \log ( l )$ with 
\begin{align*} \frac{a}{l},
\end{align*}
with $a > 0$.
\end{prop}
\begin{proof} The proof is exactly the same as \ref{ptacgmets}. One now instead obtains one higher power of $\frac{1}{(x-\alpha_0)}$ for both the model and the symplectic potential coming from the ambitoric framework. 
\end{proof}

\begin{rem} One could consider higher order vanishing as well and obtain metrics with different asymptotics near an edge with $0$ boundary measure. However, for the purposes of extremal metrics, these are the only possibilities we have to consider. In that case the $A$ and $B$ are quartics with at least two distinct zeros, and so can at most vanish to third order at one of these zeros. Note also that a third order zero can only occur in the case when two adjacent sides have $0$ boundary measure, as otherwise both zeros of $A$ or $B$ are double zeros.
\end{rem}

The model potential $\frac{1}{x}$ on $[0, \infty)$ induces the metric
\begin{align}\label{fraccuspmets} \omega = \frac{i dz \wedge d \overline{z}}{|z|^2 (- \log (|z|^2))^{\frac{3}{2}}}
\end{align}
on the unit punctured disk via the Legendre transform. Thus if one defines a space analogous to $\mathcal{S}_{P, d\sigma}$ for which the boundary behaviour is modeled on $\frac{a}{l}$ near a facet $E$ contained in the zero set of an affine linear function $l$, one obtains by similar arguments as in the Poincar\'e type case a metric with the behaviour of $\omega$ near the divisor corresponding to $E$. In the case when there are several facets with $0$ boundary measure, one can define spaces where one chooses either this or the Poincar\'e type behaviour on each such facet to get metrics with mixed cone singularities, Poincar\'e type and the behaviour of \ref{fraccuspmets} along torus invariant divisors.

Applying the results of the previous section together with this immediately gives the following result regarding extremal metrics. Below we will let $K$ be the Hessian of the function computing the Donaldson-Futaki invariant of simple piecewise linear functions with crease meeting two adjacent edges to an edge $E$ with $0$ boundary measure. 

\begin{cor}\label{metricsdescription} Suppose $(Q, d\sigma)$ is a stable weighted Delzant quadrilateral. Then $X_Q$ admits an extremal metric in $\Omega_Q$ on the complement of the torus invariant divisors corresponding to edges with $0$ weight and with cone angle singularities along the torus invariant divisors corresponding to edges with positive weight, the cone angle being prescribed by $d\sigma$. 

If $K$ is positive-definite at the point corresponding to an affine linear function with zero set containing an edge $E$ along which the boundary measure vanishes, then the metric is quasi-isometric to any order to a metric with Poincar\'e type singularities along the torus invariant divisor corresponding to $E$, whereas if $K$ is strictly positive semi-definite at this point, then the singularity along the corresponding divisor is modeled on
\begin{align*} \frac{i dz \wedge d \overline{z}}{|z|^2 (- \log (|z|^2))^{\frac{3}{2}}}.
\end{align*}
\end{cor}

Our final result is an application to the conjecture about what happens when an extremal metric does not exist. It follows from lemma \ref{subtrilem} that there are no strictly semistable weighted quadrilaterals whose unique destabilising function is a diagonal of $Q$. In fact, we get the following corollary, which shows that the conjecture of Donaldson holds in this case.

\begin{cor}\label{ssquads} Let $Q$ be a quadrilateral and suppose $d\sigma$ is a strictly semistable weight for $Q$ which is not zero at two opposite edges. Then the crease of $f$ splits $Q$ into two subpolytopes $(Q_i, d\sigma_i)$, both of which are quadrilaterals and which admit an extremal potential $u_i$ quasi-isometric to any order to an element of $\mathcal{S}_{Q_i, d\sigma_i}$.
\end{cor}
\begin{proof} As remarked above, $(Q, d\sigma)$ admits preferred ambitoric coordinates in this case. In the formal solution $H_{A,B}$, we cannot have that both $A$ and $B$ are positive, as then $(Q,d\sigma)$ would be stable. Since $\mathcal{L} (h)$ is never negative for any $h$ and $A,B$ at any interior point is a positive multiple of the Donaldson-Futaki invariant of a simple piecewise linear function, it follows that either $A$ or $B$ has a zero in the interior of their domains of definition, but that they are not negative anywhere. Say $A$ has a zero at $\alpha \in (\alpha_0, \alpha_{\infty})$. Since $A \geq 0$, it follows that $A$ must have a double zero at $x= \alpha$, unless it is exactly $0$.

In the case when $A$ is not exactly $0$, we can then restrict the ambitoric structure to $[\alpha_0, \alpha] \times [\beta_0, \beta_{\infty}]$ and $[\alpha, \alpha_{\infty}] \times [\beta_0, \beta_{\infty}]$, which in turn gives two subpolytopes of $Q$. These are quadrilaterals as the crease of $f$ is $x= \alpha$, which meets two opposite edges of $Q$. Moreover, the restriction of $A$ and $B$ to these subpolytopes give extremal potentials $u_i$ for $(Q_i, d\sigma_i)$. Since the order of vanishing at $x= \alpha$ is exactly $2$, proposition \ref{ptacgmets} implies that $u_i$ is quasi-isometric to an element of $ \mathcal{S}_{Q_i,d\sigma_i}$.

In the case when $A$ is exactly $0$, it follows in particular that the derivative of $A$ at $\alpha_0$ and $\alpha_{\infty}$ is $0$. But the derivative of $A$ at $\alpha_k$ is a positive multiple of the weight associated to the edge $\{\alpha_k\} \times [\beta_0, \beta_{\infty}]$. It follows that the boundary measure must be $0$ along the two opposite edges. 
\end{proof}

\begin{rem} Note that while we have shown that metrics with singularities modelled on \ref{fraccuspmets} can arise as solutions of the extremal equation when the boundary measure vanishes on at least one side, suggesting that these types of potentials could arise in the decomposition of a polytope into semistable subpolytopes, the above corollary shows that this does not occur for quadrilaterals.
\end{rem}

\bibliography{biblibrary}
\bibliographystyle{alpha}

\end{document}